\documentclass[dvips]{amsart}

\usepackage[foot]{amsaddr}
\title[Integrals for finite tensor categories]
{Integrals for finite tensor categories}
\author[K.~Shimizu]{Kenichi Shimizu}
\email{kshimizu@shibaura-it.ac.jp}
\address{Department of Mathematical Sciences \\
  Shibaura Institute of Technology \\
  307 Fukasaku, Minuma-ku, Saitama-shi, Saitama 337-8570, Japan.}
\thanks{The author is supported by JSPS KAKENHI Grant Number 16K17568}
\date{}

\usepackage{amsmath,amssymb,amscd}
\usepackage{upgreek}
\usepackage{bbm}
\usepackage[all]{xy}

\usepackage{amsthm}
\numberwithin{equation}{section}

\newtheorem{counter}{}[section]

\theoremstyle{definition}
\newtheorem{definition}         [counter]{Definition}

\newtheorem*{notation*}         {Notation}

\theoremstyle{plain}
\newtheorem{lemma}              [counter]{Lemma}

\newtheorem{theorem}            [counter]{Theorem}
\newtheorem{corollary}          [counter]{Corollary}
\newtheorem{question}           [counter]{Question}

\newtheorem*{theorem*}          {Theorem}
\theoremstyle{remark}
\newtheorem{remark}             [counter]{Remark}
\newtheorem{example}            [counter]{Example}

\newcommand{\id}{\mathsf{id}}
\newcommand{\eval}{\mathsf{ev}}
\newcommand{\coev}{\mathsf{coev}}
\newcommand{\proj}{{\rm pr}}
\newcommand{\op}{\mathrm{op}}
\newcommand{\rev}{\mathrm{rev}}
\newcommand{\unitobj}{\mathbbm{1}}
\newcommand{\dinto}{\xrightarrow{\ .. \ }}
\DeclareMathOperator{\Hom}{\mathrm{Hom}}
\DeclareMathOperator{\End}{\mathrm{End}}
\DeclareMathOperator{\REX}{\mathrm{Rex}}
\DeclareMathOperator{\Nat}{\mathrm{Nat}}

\DeclareMathOperator{\rank}{\mathrm{rank}}

\newcommand{\lmod}[1]{{#1}\text{-{\sf mod}}}
\newcommand{\lcomod}[1]{{#1}\text{-{\sf comod}}}

\begin{document}

\begin{abstract}
  We introduce the notions of categorical integrals and categorical cointegrals of a finite tensor category $\mathcal{C}$ by using a certain adjunction between $\mathcal{C}$ and its Drinfeld center $\mathcal{Z}(\mathcal{C})$. These notions can be identified with integrals and cointegrals of a finite-dimensional Hopf algebra $H$ if $\mathcal{C}$ is the representation category of $H$. We generalize basic results on integrals and cointegrals of a finite-dimensional Hopf algebra (such as the existence, the uniqueness, and the Maschke theorem) to finite tensor categories. Motivated by results of Lorenz, we also investigate relations between categorical integrals and morphisms factoring through projective objects. Finally, we extend the $n$-th indicator of a finite-dimensional Hopf algebra introduced by Kashina, Montgomery and Ng to finite tensor categories.
\end{abstract}

\maketitle

\section{Introduction}

A locally compact group $G$ has a left translation invariant measure $\mu$, called a Haar measure on $G$. The measure induces a linear functional $\lambda(f) = \int_G f(x) d \mu(x)$ on the space $R(G)$ of representative functions on $G$. The functional $\lambda$ inherits an invariance property from $\mu$. One can observe that the invariance property of $\lambda$ can be rephrased in terms of the coalgebra structure of $R(G)$.

Hopf algebras arise from many contexts of mathematics. One may think of a Hopf algebra as the notion of a group in the `non-commutative' geometry. From this point of view, a cointegral of a Hopf algebra is an algebraic abstraction of the functional $\lambda$ considered in the above. Formally speaking, a {\em left cointegral} of a Hopf algebra $H$ over a field $k$ is a linear map $\lambda: H \to k$ satisfying $h_{(1)} \lambda(h_{(2)}) = \lambda(h) 1_H$ for all $h \in H$, where we have used the Sweedler notation. Dually, a {\em left integral} of $H$ is an element $\Lambda \in H$ satisfying $h \Lambda = \varepsilon(h) \Lambda$ for all $h \in H$, where $\varepsilon: H \to k$ is the counit of $H$. Integrals and cointegrals are one of the central subject in the theory of Hopf algebras.

If $H$ is finite-dimensional, then it always has a non-zero left integral $\Lambda \in H$ and a non-zero left cointegral $\lambda: H \to k$ such that $\lambda(\Lambda) = 1$. This fact is crucial in the study of finite-dimensional Hopf algebras. One of the most important applications is Radford's formula on the fourth power of the antipode: There is a unique algebra map $\upalpha: H \to k$, which we call the {\em modular function} on $H$, such that
\begin{equation}
  \label{eq:intro-def-alpha}
  \Lambda h = \upalpha(h) \Lambda
\end{equation}
for all $h \in H$ ({\it cf}. the modular function on a locally compact group). Let $S$ be the antipode of $H$. Radford \cite{MR0407069} proved that $S^4$ is equal to
\begin{equation}
  \label{eq:intro-S4}
  h \mapsto \upalpha(S(h_{(1)})) \cdot h_{(2)} \cdot \upalpha(h_{(3)})
\end{equation}
modulo an inner automorphism by a certain grouplike element of $H$.

Recently, we investigate tensor categories with motivation coming from various areas of mathematics and mathematical physics. The representation category of a Hopf algebra is one of familiar examples of tensor categories. Focusing on applications of tensor categories to, for example, topological quantum field theories and conformal field theories, it is important to understand results on Hopf algebras in the context of tensor categories.

One of the significant results in this direction is a categorical analogue of the above Radford's formula. Let $\mathcal{C}$ be a finite tensor category in the sense of \cite{MR2119143}. Etingof, Nikshych and Ostrik \cite{MR2097289} proved that there is a special invertible object $\upalpha \in \mathcal{C}$ and an isomorphism
\begin{equation}
  \label{eq:intro-categroical-S4}
  X^{****} \cong \upalpha^* \otimes X \otimes \upalpha
  \quad (X \in \mathcal{C})
\end{equation}
of tensor functors. If $\mathcal{C} = \lmod{H}$ for some finite-dimensional Hopf algebra $H$, then the object $\upalpha$ corresponds to the modular function on $H$. The left-hand side and the right-hand side of \eqref{eq:intro-categroical-S4} correspond to $S^4$ and the map~\eqref{eq:intro-S4}, respectively.

To explain the aim of this paper, we recall that the modular function on a Hopf algebra is defined by using integrals as in \eqref{eq:intro-def-alpha}. On the other hand, the object $\upalpha$ of a finite tensor category is not defined in such a way: Integrals have not been defined for finite tensor categories in the first place. As many results on Hopf algebras are written in terms of integrals, it is natural to raise the following question:

\begin{question}
  What is a categorical analogue of the notion of integrals?
\end{question}

This question has been answered in \cite{2015arXiv150401178S} for {\em unimodular} finite tensor categories. In this paper, we introduce the notions of {\em categorical integrals} and {\em categorical cointegrals} of an {\em arbitrary} finite tensor category $\mathcal{C}$. As in \cite{2015arXiv150401178S}, our approach uses a right adjoint $\mathsf{R}$ of the forgetful functor $\mathsf{U}: \mathcal{Z}(\mathcal{C}) \to \mathcal{C}$ from the Drinfeld center of $\mathcal{C}$. By Tannaka reconstruction, we endow the functor $\mathsf{Z} := \mathsf{U R}$ with a structure of a Hopf comonad on $\mathcal{C}$. If $\mathcal{C} = \lmod{H}$ for some finite-dimensional Hopf algebra $H$, then the structure maps of $H$ appear as the structure morphisms of $\mathsf{Z}$. Based on this observation, categorical integrals and categorical cointegrals are defined by certain formulas written by the structure morphisms of $\mathsf{Z}$. To demonstrate that our definition is `correct', we extend some results on finite-dimensional Hopf algebras to finite tensor categories by using categorical (co)integrals; see below.

\subsection*{Organization of this paper}

Throughout this paper, we work over a fixed algebraically closed field $k$ of arbitrary characteristic. The present paper is organized as follows: In Section~\ref{sec:prelim}, we recall basic categorical notions mainly from \cite{MR1712872} and \cite{MR3242743}. Especially, the notion of an end of a functor (Subsection~\ref{subsec:ends}) is important throughout this paper.

In Section~\ref{sec:rem-adj-alg}, we collect basic results on the {\em adjoint algebra} (Definition~\ref{def:adj-alg}) and the {\em central Hopf comonad} (Definition~\ref{def:Hopf-comonad-Z}) on a finite tensor category. Let $\mathcal{C}$ be a finite tensor category over $k$. Then the forgetful functor $\mathsf{U}: \mathcal{Z}(\mathcal{C}) \to \mathcal{C}$ has a right adjoint functor, say $\mathsf{R}$. The adjoint algebra is defined by $\mathsf{A} = \mathsf{U R}(\unitobj)$. By the discussion of Section~\ref{sec:prelim}, we see that the end
\begin{equation*}
  \mathsf{Z}(V) = \int_{X \in \mathcal{C}} X \otimes V \otimes X^*
\end{equation*}
exists for all $V \in \mathcal{C}$. The functor $V \mapsto \mathsf{Z}(V)$ has a structure of a Hopf comonad on $\mathcal{C}$, which we call the central Hopf comonad, such that the category of $\mathsf{Z}$-comodules can be identified with $\mathcal{Z}(\mathcal{C})$.  Under this identification, $\mathsf{R}$ is just the free $\mathsf{Z}$-comodule functor. This observation allows us to study the adjoint algebra by the central Hopf comonad. The results of this section are important for studying our categorical analogue of integrals in later sections. However, since the detail is technical, we omit to describe them here.

At the end of this section, we give a summary of notations introduced in this section for reader's convenience. These notations will be used throughout the rest of this paper. Finally, we observe that the structure morphisms of $\mathsf{Z}$,
\begin{equation}
  \label{eq:Z-structures}
  \begin{gathered}
    \mathsf{m}_{V,W}: \mathsf{Z}(V) \otimes \mathsf{Z}(W) \to \mathsf{Z}(V \otimes W),
    \quad \mathsf{u}: \unitobj \to \mathsf{Z}(\unitobj), \\
    \updelta_V: \mathsf{Z}(V) \to \mathsf{Z}^2(V),
    \quad \upepsilon_V: \mathsf{Z}(V) \to V,
  \end{gathered}
\end{equation}
closely relate to the structure morphisms of $H$ if $\mathcal{C} = \lmod{H}$ is the representation category of a finite-dimensional Hopf algebra $H$ (Example~\ref{ex:adj-Hopf-alg}).

In Section~\ref{sec:integ-ftc}, we first introduce the {\em modular object} $\upalpha \in \mathcal{C}$ (Definition~\ref{def:ftc-mod-ft}) by using the adjunction $\mathsf{U} \dashv \mathsf{R}$. We then introduce {\em categorical integrals} and {\em categorical cointegrals} of $\mathcal{C}$ (Definition~\ref{def:integ-FTCs}). A categorical integral of $\mathcal{C}$ is a morphism $\Lambda: \unitobj \to \mathsf{Z}(\upalpha)$ in $\mathcal{C}$ satisfying a certain condition written by the structure morphisms \eqref{eq:Z-structures} of $\mathsf{Z}$. A categorical cointegral is a morphism $\lambda: \unitobj \to \mathsf{Z}(\upalpha)$ in $\mathcal{C}$ satisfying a certain condition which is also written by \eqref{eq:Z-structures}. One can check that a categorical (co)integral is an ordinary (co)integral if $\mathcal{C} = \lmod{H}$ for some finite-dimensional Hopf algebra $H$ (Example~\ref{ex:Hopf-integ}).

After defining categorical (co)integrals, we give some results on them: We first prove that the space of the dimension of categorical (co)integrals is one-dimensional. Moreover, for any non-zero categorical cointegral $\lambda$, there is a unique categorical integral $\Lambda$ such that $\lambda \circ \Lambda = \id_{\unitobj}$ (Theorem~\ref{thm:integ-exist}). The proof of these results depends on the representation theory of the adjoint algebra given in Section~\ref{sec:rem-adj-alg}. The Maschke theorem is also established (Theorem~\ref{thm::Maschke}). Finally, we introduce the property CU of a finite tensor category (Definition~\ref{def:FTC-CU}). Since a finite-dimensional Hopf algebra $H$ is cosemisimple unimodular if and only if $\lmod{H}$ has the property CU (Theorem~\ref{thm:FTC-CU}), we conclude that the cosemisimple unimodularity is a gauge invariant of a finite-dimensional Hopf algebra, as has been proved in \cite{MR1939116}.

In Section~\ref{sec:integ-func}, we introduce the {\em integration functor} $I: \mathcal{C} \to \mathcal{C}$. This functor is a categorical analogue of the submodule $I(X) = \{ \Lambda x \mid x \in X \}$ of $X \in \lmod{H}$, where $\Lambda$ is a non-zero left integral of $H$. As in the case of Hopf algebras, the object $I(X)$ for $X \in \mathcal{C}$ is isomorphic to the direct sum of the tensor unit $\unitobj$. Generalizing a result of Lorenz \cite{MR1961567}, we show that there is an isomorphism
\begin{equation*}
  \Hom_{\mathcal{C}}(\unitobj, I(W \otimes V^*))
  \cong \{ f: V \to W \mid \text{$f$ factors through a projective object} \}
\end{equation*}
for all $V, W \in \mathcal{C}$ (Theorem~\ref{thm:projmorph}). We also prove that $\dim_k \Hom_{\mathcal{C}}(\unitobj, I(\mathsf{A}))$, where $\mathsf{A}$ is the adjoint algebra of $\mathcal{C}$, is equal to the rank of the Cartan matrix of $\mathcal{C}$ reduced modulo the characteristic of $k$ (see Theorem~\ref{thm:integ-higman} for the precise statement). This theorem extends the result of Lorenz \cite{MR1435369} to finite tensor categories.

In Section~\ref{sec:indicator-ftc}, we define the $n$-th indicator $\nu_n(\mathcal{C}) \in k$ for a positive integer $n$ by using categorical integrals and categorical cointegrals. Kashina, Montgomery and Ng introduced the $n$-th indicator $\nu_n(H) \in k$ of a finite-dimensional Hopf algebra $H$ and proved that $\nu_n$ is gauge invariant \cite{KMN09}. Let $\lambda$ be a non-zero left cointegral of $H$, and let $\Lambda$ be a non-zero integral of $H$ such that $\lambda(\Lambda) = 1$. As they proved in  \cite{KMN09}, then $\nu_n(H) = \lambda(\Lambda^{[n]})$, where $(-)^{[n]}$ means the $n$-th Sweedler power of an element of $H$. In a similar way as the $n$-th Frobenius-Schur endomorphism introduced in \cite{MR2381536}, we define the $n$-th Sweedler power $\Lambda^{[n]}$ of a categorical left integral $\Lambda$ of $\mathcal{C}$ (Definition~\ref{def:sw-pow-integ}). We then define $\nu_n(\mathcal{C}) \in k$ by imitating the above-mentioned formula of $\nu_n(H)$ by integrals (Definition~\ref{def:indicator-FTC}). It turns out that 
\begin{equation*}
  \nu_n(\lmod{H}) = \nu_n(H)
\end{equation*}
for all positive integers $n$ (Theorem~\ref{thm:indicator-FTC-Hopf}). We hope that this categorical description of the $n$-th indicator of $H$ give some new insights for future study of $\nu_n(H)$ and other gauge invariants of finite-dimensional Hopf algebras.

\subsection*{Acknowledgment}

The author is supported by JSPS KAKENHI Grant Number 16K17568.

\section{Preliminaries}
\label{sec:prelim}

\subsection{Monoidal categories}

For the basic theory on monoidal categories, we refer the reader to \cite{MR1712872} and \cite{MR3242743}. In view of Mac Lane's coherence theorem, we assume all monoidal categories to be strict. Given a monoidal category $\mathcal{C} = (\mathcal{C}, \otimes, \unitobj)$ with tensor product $\otimes$ and unit object $\unitobj \in \mathcal{C}$, we set
\begin{equation*}
  \mathcal{C}^{\op} = (\mathcal{C}^{\op}, \otimes, \unitobj)
  \text{\quad and \quad}
  \mathcal{C}^{\rev} = (\mathcal{C}, \otimes^{\rev}, \unitobj),
\end{equation*}
where $\otimes^{\rev}$ is the reversed tensor product $X \otimes^{\rev} Y = Y \otimes X$.

Let $L$ and $R$ be objects of a monoidal category $\mathcal{C}$, and let $\varepsilon: L \otimes R \to \unitobj$ and $\eta: \unitobj \to R \otimes L$ be morphisms in $\mathcal{C}$. Following \cite{MR3242743}, we say that $(L, \varepsilon, \eta)$ is a {\em left dual object} of $R$ and $(R, \varepsilon, \eta)$ is a {\em right dual object} of $L$ if the equations
\begin{equation*}
  (\varepsilon \otimes \id_L) \circ (\id_L \otimes \eta) = \id_L
  \quad \text{and} \quad
  (\id_R \otimes \varepsilon) \circ (\eta \otimes \id_R) = \id_R
\end{equation*}
hold. The monoidal category $\mathcal{C}$ is {\em rigid} if every object of $\mathcal{C}$ has a fixed left dual object and a fixed right dual object. If this is the case, we denote the left dual object and the right dual object of $X \in \mathcal{C}$ by
\begin{equation*}
  (X^*, \eval_X, \coev_X) \quad \text{and} \quad
  ({}^*\!X, \eval_X', \coev_X'),
\end{equation*}
respectively. The assignment $X \mapsto X^*$ extends to a monoidal equivalence between $\mathcal{C}^{\rev}$ and $\mathcal{C}^{\op}$, which we call the {\em left duality functor}. A quasi-inverse of this equivalence is given by $X \mapsto {}^* \! X$. We may assume that $X \mapsto X^*$ and $X \mapsto {}^* \! X$ are strict monoidal functors and mutually inverse to each other. Thus, in particular,
\begin{equation*}
  (X \otimes Y)^* = Y^* \otimes X^*,
  \quad \unitobj^* = \unitobj
  \quad \text{and}
  \quad {}^*(X^*) = X = ({}^* \! X)^*
\end{equation*}
for all objects $X$ and $Y$ of a rigid monoidal category.

\subsection{Hopf comonads}

Let $\mathcal{C}$ be a monoidal category. A {\em monoidal comonad} on $\mathcal{C}$ is a comonad $T = (T, \delta, \varepsilon)$ on $\mathcal{C}$ such that the functor $T: \mathcal{C} \to \mathcal{C}$ is a monoidal functor and $\delta$ and $\varepsilon$ are monoidal natural transformations. If $T$ is a monoidal comonad on $\mathcal{C}$ with the monoidal structure
\begin{equation*}
  \mu_{X,Y}: T(X) \otimes T(Y) \to T(X \otimes Y)
  \quad \text{and} \quad \eta: \unitobj \to T(\unitobj),
\end{equation*}
then the category $\lcomod{T}$ of $T$-comodules ($=$ the Eilenberg-Moore category of the comonad $T$) is a monoidal category with the tensor product given by
\begin{equation*}
  (M, \rho) \otimes (M', \rho') = (M \otimes M', \mu_{M,M'} \circ (\rho \otimes \rho'))
\end{equation*}
for $(M, \rho), (M', \rho') \in \lcomod{T}$ and the unit object $(\unitobj, \eta)$. We now assume that $\mathcal{C}$ is rigid. Then a {\em Hopf comonad} on $\mathcal{C}$ is a monoidal comonad $T$ on $\mathcal{C}$ such that the monoidal category $\lcomod{T}$ is rigid.

A monoidal comonad on $\mathcal{C}$ introduced in the above is precisely a bimonad on $\mathcal{C}^{\op}$ in the sense of \cite{MR2355605}. A Hopf comonad on $\mathcal{C}$ is just a Hopf monad on $\mathcal{C}^{\op}$ in the sense of \cite{MR2355605}. For the general theory of monoidal comonads and Hopf comonads, we refer the reader to \cite{MR2355605,MR2793022,MR2869176}, where bimonads and Hopf monads are investigated.

\subsection{Module categories}

Let $\mathcal{C}$ be a monoidal category. A {\em left $\mathcal{C}$-module category} \cite[Chapter 7]{MR3242743} is a category $\mathcal{M}$ endowed with a functor $\otimes: \mathcal{C} \times \mathcal{M} \to \mathcal{M}$, called the {\em action}, and natural isomorphisms
\begin{equation*}
  (X \otimes Y) \otimes M \cong X \otimes (Y \otimes M)
  \text{\quad and \quad}
  \unitobj \otimes M \cong M
  \quad (X, Y \in \mathcal{C}, M \in \mathcal{M})
\end{equation*}
satisfying certain axioms similar to those for a monoidal category. In a similar way, right $\mathcal{C}$-module categories and $\mathcal{C}$-bimodule categories are defined.

Let $\mathcal{M}$ be a left $\mathcal{C}$-module category, and let $A$ be an algebra in $\mathcal{C}$ ($=$ a monoid in $\mathcal{C}$ \cite{MR1712872}). Then the functor $M \mapsto A \otimes M$ ($M \in \mathcal{M}$) has a natural structure of a monad on $\mathcal{M}$. We denote the Eilenberg-Moore category of this monad by ${}_A \mathcal{M}$ and call it {\em the category of $A$-modules in $\mathcal{M}$}. If $\mathcal{M}$ is a right $\mathcal{C}$-module category, then the category $\mathcal{M}_A$ of right $A$-modules is defined analogously

Suppose that $\mathcal{C}$ is rigid. Since $\mathcal{C}$ is a $\mathcal{C}$-bimodule category by the tensor product, the categories ${}_A \mathcal{C}$ and $\mathcal{C}_A$ are defined. If $(M, \rho)$ is a left $A$-module in $\mathcal{C}$, then $M^*$ is a right $A$-module in $\mathcal{C}$ by the action $\rho': M^* \otimes A \to M^*$ determined by
\begin{equation*}
  \eval_{M} \circ (\id_{M^*} \otimes \rho) = \eval_{M} \circ (\rho' \otimes \id_M).
\end{equation*}
Similarly, if $M$ is a right $A$-module in $\mathcal{C}$, then its right dual object ${}^* M$ is a left $A$-module in a natural way. These constructions give anti-equivalences
\begin{equation*}
  (-)^*: {}_A \mathcal{C} \to \mathcal{C}_A
  \quad \text{and} \quad
  {}^*(-): \mathcal{C}_A \to {}_A \mathcal{C}
\end{equation*}
of categories, that are mutually inverse to each other. In particular, $A^*$ is a right $A$-module. For later use, we remark the following lemma:

\begin{lemma}
  \label{lem:module-adj}
  Let $\mathcal{M}$ be a right module category over a rigid monoidal category $\mathcal{C}$, and let $A$ be an algebra in $\mathcal{C}$. Then the functor
  \begin{equation*}
    \id_{\mathcal{M}} \otimes A^*: \mathcal{M} \to \mathcal{M}_A,
    \quad M \mapsto M \otimes A^*
  \end{equation*}
  is right adjoint to the forgetful functor $\mathcal{M}_A \to \mathcal{M}$.
\end{lemma}

The proof is standard and is given, for example, in \cite[Lemma 2.1]{2014arXiv1402.3482S}. There is a similar description of a right adjoint of the forgetful functor from the category of left modules.

\subsection{Finite tensor categories}

Throughout this paper, we work over a fixed algebraically closed field $k$. By an algebra over $k$, we always mean an associative and unital algebra over the field $k$. Given an algebra $A$ over $k$, we denote by $\lmod{A}$ the category of finite-dimensional left $A$-modules.

\begin{definition}
  A {\em finite abelian category} (over $k$) is a $k$-linear category that is equivalent to $\lmod{A}$ for some finite-dimensional algebra $A$ over $k$.
\end{definition}

A $k$-linear functor between finite abelian categories is left (right) exact if and only if it has a left (right) adjoint. Now we set $\mathcal{V} = \lmod{k}$. By using this fact, we define the left action $\bullet$ of $\mathcal{V}$ on a finite abelian category $\mathcal{M}$ by
\begin{equation*}
  \Hom_{\mathcal{M}}(V \bullet X, Y) \cong \Hom_{k}(V, \Hom_{\mathcal{M}}(X, Y))
\end{equation*}
for $V \in \mathcal{V}$ and $X, Y \in \mathcal{M}$.

\begin{definition}
  A {\em finite tensor category} \cite{MR2119143} is a rigid monoidal category $\mathcal{C}$ such that $\mathcal{C}$ is a finite abelian category over $k$, the tensor product of $\mathcal{C}$ is $k$-linear in each variable, and the unit object $\unitobj \in \mathcal{C}$ is a simple object.
\end{definition}

Let $\mathcal{C}$ be a finite tensor category. Then the tensor product of $\mathcal{C}$ is exact and faithful in each variable. It is known that $\mathcal{C}$ is {\em Frobenius} in the sense that an object of $\mathcal{C}$ is injective if and only if it is projective. If $P \in \mathcal{C}$ is projective, then $P \otimes X$ is projective for all objects $X \in \mathcal{C}$.

\subsection{Ends and coends}
\label{subsec:ends}

Let $\mathcal{A}$ and $\mathcal{V}$ be categories, and let $S, T: \mathcal{A}^{\op} \times \mathcal{A} \to \mathcal{V}$ be functors. A {\em dinatural transformation} $\omega: S \dinto T$ from $S$ to $T$ is a family
\begin{equation*}
  \omega = \{ \omega_X: S(X, X) \to T(X, X) \}_{X \in \mathcal{A}}
\end{equation*}
of morphisms in $\mathcal{A}$ such that the equation
\begin{equation*}
  T(X, f) \circ \omega_X \circ S(f, X)
  = T(f, Y) \circ \omega_Y \circ S(Y, f)
\end{equation*}
holds for all morphisms $f: X \to Y$ in $\mathcal{A}$. An {\em end} of $T$ is a pair $(E, \pi)$ consisting of an object $E \in \mathcal{V}$ and a dinatural transformation $\pi: E \dinto T$ from the object $E$ (regarded as a constant functor from $\mathcal{A}^{\op} \times \mathcal{A}$ to $\mathcal{V}$) to $T$ that is `universal' in the following sense: For any pair $(M, \omega)$ consisting of an object $M \in \mathcal{V}$ and a dinatural transformation $\omega: M \dinto T$, there exists a unique morphism $\phi: M \to E$ such that $\omega_X = \pi_X \circ \phi$ for all $X \in \mathcal{A}$. A {\em coend} of $S$ is a pair $(C, i)$ consisting of an object $C \in \mathcal{V}$ and a `universal' dinatural transformation $i: S \dinto C$. Following \cite{MR1712872}, the end of $T$ and the coend of $S$ will be denoted by
\begin{equation*}
  \int_{X \in \mathcal{A}} T(X, X)
  \quad \text{and} \quad
  \int^{X \in \mathcal{A}} S(X, X),
\end{equation*}
respectively.

In this paper, the following formula will be used extensively: Let $F, G: \mathcal{A} \to \mathcal{B}$ be functors from an essentially small category $\mathcal{A}$ to a category $\mathcal{B}$. Let $\Nat(F, G)$ denote the set of natural transformations from $F$ to $G$. Then we have
\begin{equation*}
  \int_{X \in \mathcal{A}} \Hom_{\mathcal{B}}(F(X), G(X)) \cong \Nat(F, G)
\end{equation*}
in the category of sets; see \cite[IX.5]{MR1712872}.

\subsection{Eilenberg-Watts calculus}
\label{subsec:EW-calc}

For finite abelian categories $\mathcal{M}$ and $\mathcal{N}$, we denote by $\REX(\mathcal{M}, \mathcal{N})$ the category of $k$-linear right exact functors from $\mathcal{M}$ to $\mathcal{N}$. We also denote by $\mathcal{M} \boxtimes \mathcal{N}$ their Deligne tensor product \cite{MR1106898}; see also \cite[Subsection 1.11]{MR3242743}. The Eilenberg-Watts theorem implies the following result:

\begin{lemma}
  \label{lem:EW-calc}
  For finite abelian categories $\mathcal{M}$ and $\mathcal{N}$, the functor
  \begin{equation}
    \label{eq:EX-calc-equiv-1}
    \mathcal{M}^{\op} \boxtimes \mathcal{N} \to \REX(\mathcal{M}, \mathcal{N}),
    \quad M^{\op} \boxtimes N \mapsto \Hom_{\mathcal{M}}(-, M)^* \bullet N
  \end{equation}
  is an equivalence of $k$-linear categories. Here, an object $M \in \mathcal{M}$ is denoted by $M^{\op}$ when it is regarded as an object of $\mathcal{M}^{\op}$. Moreover, the end
  \begin{equation*}
    \int_{X \in \mathcal{M}} X^{\op} \boxtimes F(X) \in \mathcal{M}^{\op} \boxtimes \mathcal{N}
  \end{equation*}
  exists for all $k$-linear right exact functors $F: \mathcal{M} \to \mathcal{N}$, and the assignment
  \begin{equation}
    \label{eq:EX-calc-equiv-2}
    \REX(\mathcal{M}, \mathcal{N}) \to \mathcal{M}^{\op} \boxtimes \mathcal{N},
    \quad F \mapsto \int_{X \in \mathcal{M}} X^{\op} \boxtimes F(X)
  \end{equation}
  is a quasi-inverse functor of \eqref{eq:EX-calc-equiv-1}.
\end{lemma}

A detailed proof of this lemma is found in \cite[Lemma 3.3]{2014arXiv1412.0211S}. See also \cite{2016arXiv161204561F} for related results and other variants of this type of equivalences.

Lemma~\ref{lem:EW-calc} implies that $\REX(\mathcal{M}, \mathcal{N})$ is a finite abelian category. As an application of the lemma, we give the following characterization of the projectivity of an object of $\REX(\mathcal{M}, \mathcal{N})$.

\begin{lemma}
  \label{lem:Rex-proj}
  Let $\mathcal{M}$ and $\mathcal{N}$ be finite abelian categories, and let $F: \mathcal{M} \to \mathcal{N}$ be a $k$-linear right exact functor with right adjoint $G$. The functor $F$ is a projective object of $\REX(\mathcal{M}, \mathcal{N})$ if and only if
  \begin{enumerate}
  \item $F(M) \in \mathcal{N}$ is projective for all objects $M \in \mathcal{M}$, and
  \item $G(N) \in \mathcal{M}$ is injective for all objects $N \in \mathcal{N}$.
  \end{enumerate}
\end{lemma}
\begin{proof}
  An object $F \in \REX(\mathcal{M}, \mathcal{N})$ is projective if and only if the functor
  \begin{equation*}
    \mathcal{Y}_F := \Nat(F, -) \circ \Phi
  \end{equation*}
  is exact, where $\Phi$ is the equivalence given by~\eqref{eq:EX-calc-equiv-1}. By the property of the Deligne tensor product \cite[Proposition~5.13]{MR1106898}, $\mathcal{Y}_F$ is exact if and only if the functor
  \begin{equation}
    \label{eq:Rex-proj-pf-1}
    \mathcal{M}^{\op} \times \mathcal{N} \to \mathcal{V},
    \quad (M^{\op}, N) \mapsto \mathcal{Y}_F(M^{\op} \boxtimes N)
  \end{equation}
  is exact in each variable. For $M \in \mathcal{M}$ and $N \in \mathcal{N}$, we have:
  \begin{align*}
    \mathcal{Y}_F(M^{\op} \boxtimes N)
    & = \Nat(F, \Phi(M^{\op} \boxtimes N)) \\
    & \cong \textstyle \int_{X \in \mathcal{M}} \Hom_{\mathcal{N}}(F(X), \Hom_{\mathcal{M}}(X, M)^* \bullet N) \\
    & \cong \textstyle \int_{X \in \mathcal{M}} \Hom_{\mathcal{M}}(X, M)^* \otimes_k \Hom_{\mathcal{N}}(F(X), N) \\
    & \cong \textstyle \int_{X \in \mathcal{M}}
      \Hom_k(\Hom_{\mathcal{M}}(X, M), \Hom_{\mathcal{N}}(F(X), N)) \\
    & \cong \Nat(\Hom_{\mathcal{M}}(-, M), \Hom_{\mathcal{N}}(F(-), N))) \\
    & \cong \Hom_{\mathcal{N}}(F(M), N) \quad \text{(by the Yoneda lemma)} \\
    & \cong \Hom_{\mathcal{M}}(M, G(N)).
  \end{align*}
  Hence the exactness of~\eqref{eq:Rex-proj-pf-1} in the variable $N$ is equivalent to the condition (1), and the exactness in $M$ is equivalent to (2). The proof is done.
\end{proof}

We will use the above lemma in the following form:

\begin{corollary}
  \label{cor:Rex-proj}
  Let $\mathcal{C}$ be a finite tensor category, and let $X \in \mathcal{C}$ be an object. Then the functor $X \otimes \id_{\mathcal{C}}$ is a projective object of $\REX(\mathcal{C}, \mathcal{C})$ if and only if $X$ is a projective object of $\mathcal{C}$.
\end{corollary}

\section{Remarks on the adjoint algebra}
\label{sec:rem-adj-alg}

\subsection{The adjoint algebra}
\label{subsec:Dri-cent}

Let $\mathcal{C}$ be a finite tensor category over $k$, and let $V \in \mathcal{C}$ be an object. A {\em half-braiding} for $V$ is a natural isomorphism
\begin{equation*}
  \sigma_X: V \otimes X \to X \otimes V
  \quad (X \in \mathcal{C})
\end{equation*}
satisfying $\sigma_{X \otimes Y} = (\id_X \otimes \sigma_Y) \circ (\sigma_X \otimes \id_Y)$ for all objects $X, Y \in \mathcal{C}$. The {\em Drinfeld center} of $\mathcal{C}$, denoted by $\mathcal{Z}(\mathcal{C})$, is the category whose objects are pairs $(V, \sigma)$ consisting of an object $V \in \mathcal{C}$ and a half-braiding $\sigma$ for $V$. The Drinfeld center is in fact a braided finite tensor category; see \cite[Proposition 7.13.8]{MR3242743}.

There is the forgetful functor $\mathsf{U}: \mathcal{Z}(\mathcal{C}) \to \mathcal{C}$ defined by $(V, \sigma) \mapsto V$. It is known that $\mathsf{U}$ has a right adjoint, say $\mathsf{R}$. Since $\mathsf{U}$ is a strong monoidal functor, $\mathsf{R}$ is a (lax) monoidal functor. Thus $\mathsf{A} := \mathsf{U} \mathsf{R}(\unitobj)$ is an algebra in $\mathcal{C}$ as the image of the trivial algebra under the monoidal functor $\mathsf{U} \mathsf{R}$.

\begin{definition}
  \label{def:adj-alg}
  We call the algebra $\mathsf{A}$ the {\em adjoint algebra} of $\mathcal{C}$.
\end{definition}

As we will see in Example~\ref{ex:adj-Hopf-alg}, the adjoint algebra is a categorical analogue of the adjoint representation of a Hopf algebra, which is in fact a module-algebra over the Hopf algebra. The adjoint algebra plays a crucial role in recent study of finite tensor categories; see \cite{MR3314297,2015arXiv150401178S,2016arXiv160805905S}. In this section, we summarize results on the adjoint algebra and note some immediate consequence of them.

\subsection{The central Hopf comonad}

Let $\mathcal{C}$ be a finite tensor category. There is the following sequence of $k$-linear exact functors:
\begin{equation*}
  \REX(\mathcal{C})
  := \REX(\mathcal{C}, \mathcal{C})
  \xrightarrow[{\quad \text{Lemma~\ref{lem:EW-calc}} \quad}]{\approx}
  \mathcal{C} \boxtimes \mathcal{C}^{\op}
  \xrightarrow[{\quad \id_{\mathcal{C}} \boxtimes (-)^* \quad}]{\approx}
  \mathcal{C} \boxtimes \mathcal{C}
  \xrightarrow[{\quad \otimes \quad }]{} \mathcal{C}.
\end{equation*}
By considering the image of $\id_{\mathcal{C}} \otimes V \in \REX(\mathcal{C})$, we see that the end
\begin{equation}
  \label{eq:Z-end}
  \mathsf{Z}(V) = \int_{X \in \mathcal{C}} X \otimes V \otimes X^*
\end{equation}
exists for all $V \in \mathcal{C}$. We denote by
\begin{equation}
  \uppi_{V;X}: \mathsf{Z}(V) \to X \otimes V \otimes X^* \quad (V, X \in \mathcal{C})
\end{equation}
the universal dinatural transformation for the end $\mathsf{Z}(V)$. The assignment $V \mapsto \mathsf{Z}(V)$ extends to a $k$-linear endofunctor on $\mathcal{C}$ in such a way that the universal dinatural transformation $\uppi_{V;X}$ is natural in $V$. The functor $Z$ has a comonad structure
\begin{equation}
  \label{eq:Z-comonad}
  \updelta_V: \mathsf{Z}(V) \to \mathsf{Z}^2(V)
  \quad \text{and} \quad
  \upepsilon_V: \mathsf{Z}(V) \to V
  \quad (V \in \mathcal{C})
\end{equation}
uniquely determined by the property that the equations $\upepsilon_V = \uppi_{\unitobj; V}$ and
\begin{equation}
  \label{eq:Z-def-comult}
  (\id_{X} \otimes \uppi_{V;Y} \otimes \id_{X^*}) \circ \uppi_{\mathsf{Z}(V);X} \circ \updelta_V = \uppi_{V; X \otimes Y}
\end{equation}
hold for all objects $V, X, Y \in \mathcal{C}$ (see \cite[Subsection 3.4]{2016arXiv160805905S} for details).

The category $\lcomod{\mathsf{Z}}$ of $\mathsf{Z}$-comodules can be identified with the Drinfeld center of $\mathcal{C}$ ({\it cf}. Day-Street \cite{MR2342829} and Brugui\`eres-Virelizier \cite{MR2869176}, where the dual situation is considered). For simplicity, we set
\begin{equation}
  \label{eq:Z-partial}
  \partial_{V,X} = (\id_X \otimes \id_V \otimes \eval_X) \circ (\uppi_{V; X} \otimes \id_X)
\end{equation}
for $V, X \in \mathcal{C}$. To establish the identification $\lcomod{\mathsf{Z}} \cong \mathcal{Z}(\mathcal{C})$, we note:

\begin{lemma}
  \label{lem:Hom-V-ZW}
  Given a morphism $t: V \to \mathsf{Z}(W)$ in $\mathcal{C}$, we define
  \begin{equation}
    \label{eq:Hom-V-ZW}
    \widetilde{t}_X = \partial_{W,X} \circ (t \otimes \id_X)
  \end{equation}
  for $X \in \mathcal{C}$. Then $\widetilde{t}_X$ is natural in the variable $X \in \mathcal{C}$. Moreover, the map
  \begin{equation}
    \label{eq:Hom-V-ZW-iso}
    \Hom_{\mathcal{C}}(V, \mathsf{Z}(W)) \to \Nat(V \otimes \id_{\mathcal{C}}, \id_{\mathcal{C}} \otimes W),
    \quad t \mapsto \widetilde{t}
  \end{equation}
  is an isomorphism that is natural in the variables $V, W \in \mathcal{C}$.
\end{lemma}
\begin{proof}
  The isomorphism~\eqref{eq:Hom-V-ZW-iso} is given by the composition of isomorphisms
  \begin{align*}
    \Hom_{\mathcal{C}}(V, \mathsf{Z}(W))
    & \xrightarrow{\quad \cong \quad} \textstyle \int_{X \in \mathcal{C}} \Hom_{\mathcal{C}}(V, X \otimes W \otimes X^*) \\
    & \xrightarrow{\quad \cong \quad} \textstyle \int_{X \in \mathcal{C}} \Hom_{\mathcal{C}}(V \otimes X, X \otimes W) \\
    & \xrightarrow{\quad \cong \quad} \Nat(V \otimes \id_{\mathcal{C}}, \id_{\mathcal{C}} \otimes W).
      \qedhere
  \end{align*}
\end{proof}

One can check that a morphism $\rho: V \to \mathsf{Z}(V)$ in $\mathcal{C}$ makes $V$ a $\mathsf{Z}$-comodule if and only if the natural transformation $\widetilde{\rho}: V \otimes \id_{\mathcal{C}} \to \id_{\mathcal{C}} \otimes V$ corresponding to $\rho$ is a half-braiding for $V$. Hence we obtain a category isomorphism
\begin{equation}
  \label{eq:Z-comod-ZC}
  \lcomod{\mathsf{Z}} \cong \mathcal{Z}(\mathcal{C}),
  \quad (V, \rho) \mapsto (V, \widetilde{\rho}\,).
\end{equation}

An object of the form $\mathsf{Z}(W)$, $W \in \mathcal{C}$, is a $\mathsf{Z}$-comodule with the coaction given by the comultiplication of $\mathsf{Z}$. Thus there is a half-braiding for $\mathsf{Z}(W)$, which we will refer to as the {\em canonical half-braiding}. It is convenient to express the natural transformation given by \eqref{eq:Hom-V-ZW} by the canonical half-braiding.

\begin{lemma}
  \label{lem:Hom-V-ZW-br}
  Let $t: V \to \mathsf{Z}(W)$ be a morphism in $\mathcal{C}$, and let $\widetilde{t}$ be the corresponding natural transformation. Then we have
  \begin{equation}
    \label{eq:Hom-V-ZW-br}
    \widetilde{t}_X = (\id_X \otimes \upepsilon_{W}) \circ \sigma_{X} \circ (t \otimes \id_X)
  \end{equation}
  for all $X \in \mathcal{C}$, where $\sigma$ is the canonical half-braiding for $\mathsf{Z}(W)$.
\end{lemma}
\begin{proof}
  By definition, $\sigma$ corresponds to $\updelta_W: \mathsf{Z}(W) \to \mathsf{Z}^2(W)$ via~\eqref{eq:Hom-V-ZW-iso}. Thus,
  \begin{align*}
    (\id_X \otimes \upepsilon_{W}) \circ \sigma_{X}
    & = (\id_X \otimes \upepsilon_{W}) \circ \partial_{\,\mathsf{Z}(W), X} \circ (\updelta_{W} \otimes \id_X) \\
    & = \partial_{W, X} \circ (\upepsilon_{W} \otimes \id_X) \circ (\updelta_{W} \otimes \id_X) \\
    & = \partial_{W,X} \qquad \text{(by the counit axiom)}
  \end{align*}
  for all objects $W, X \in \mathcal{C}$. Now~\eqref{eq:Hom-V-ZW-br} follows from~\eqref{eq:Hom-V-ZW}.
\end{proof}

The isomorphism \eqref{eq:Z-comod-ZC} commutes with the forgetful functors to $\mathcal{C}$. Since $\mathcal{Z}(\mathcal{C})$ is rigid, the comonad $\mathsf{Z}$ has a structure of a Hopf comonad on $\mathcal{C}$ such that \eqref{eq:Z-comod-ZC} is an isomorphism of monoidal categories. The monoidal structure
\begin{equation}
  \label{eq:Z-monoidal}
  \mathsf{u}: \unitobj \to \mathsf{Z}(\unitobj)
  \quad \text{and} \quad
  \mathsf{m}_{V,W}: \mathsf{Z}(V) \otimes \mathsf{Z}(W) \to \mathsf{Z}(V \otimes W)
\end{equation}
of $\mathsf{Z}$ are the morphisms uniquely determined by $\uppi_{\unitobj; X} \circ \mathsf{u} = \coev_X$ and
\begin{equation}
  \label{eq:Z-def-mult}
  \uppi_{V \otimes W; X} \circ \mathsf{m}_{V,W}
  = (\id_X \otimes \id_V \otimes \eval_X \otimes \id_W \otimes \id_{X^*}) \circ (\uppi_{V;X} \otimes \uppi_{W;X})
\end{equation}
for all objects $V, W, X \in \mathcal{C}$, respectively (see \cite{MR2869176}, where the quantum double of a Hopf monad is described).

\begin{definition}
  \label{def:Hopf-comonad-Z}
  We call $\mathsf{Z}$ the {\em central Hopf comonad} on $\mathcal{C}$.
\end{definition}

\subsection{Representation theory of the adjoint algebra}

Let $\mathcal{C}$ be a finite tensor category. We use the same notation as in the previous subsection. If we identify $\mathcal{Z}(\mathcal{C})$ with $\lcomod{\mathsf{Z}}$ by \eqref{eq:Z-comod-ZC}, then the free $\mathsf{Z}$-comodule functor
\begin{equation}
  \label{eq:free-Z-comod}
  \mathsf{R}: \mathcal{C} \to \mathcal{Z}(\mathcal{C}),
  \quad V \mapsto (\mathsf{Z}(V), \updelta_V)
\end{equation}
is a right adjoint functor of $\mathsf{U}$. The unit $\id_{\mathcal{C}} \to \mathsf{U} \mathsf{R}$ and the counit $\mathsf{R} \mathsf{U} \to \id_{\mathcal{Z}(\mathcal{C})}$ of the adjunction $\mathsf{U} \dashv \mathsf{R}$ are given by the comonad structure of $\mathsf{Z}$. The monoidal structure of $\mathsf{R}$ is the morphisms
\begin{equation*}
  \mathsf{m}_{V,W}: \mathsf{R}(V) \otimes \mathsf{R}(W) \to \mathsf{R}(V \otimes W)
  \quad \text{and} \quad \mathsf{u}: \unitobj \to \mathsf{R}(\unitobj),
\end{equation*}
that is, the same structure morphisms as $\mathsf{Z}$. By the definition $\mathsf{U}$ and $\mathsf{R}$, we see that the adjoint algebra is the object
\begin{equation}
  \label{eq:adj-alg-Z}
  \mathsf{A} = \mathsf{U R}(\unitobj) = \mathsf{Z}(\unitobj) = \int_{X \in \mathcal{C}} X \otimes X^*
\end{equation}
with multiplication $\mathsf{m}_{\unitobj,\unitobj}$ and unit $\mathsf{u}$.

\begin{definition}
  \label{def:can-act}
  For an object $X \in \mathcal{C}$, we define
  \begin{equation}
    \uprho^{\ell}_X =\partial_{\unitobj, X}
    \quad \text{and} \quad
    \uprho^{r}_X = (\eval'_{X} \otimes \id_X) \circ (\id_X \otimes \uppi_{\,{}^*\!X; \unitobj}),
  \end{equation}
  where $\partial$ is the natural transformation given by~\eqref{eq:Z-partial}. We call $\uprho^{\ell}_X$ and $\uprho^r_X$ the {\em left} and the {\em right canonical action} of $\mathsf{A}$ on $X$, respectively.
\end{definition}

The morphisms $\uprho^{\ell}_X: \mathsf{A} \otimes X \to X$ and $\uprho^{r}_X: X \otimes \mathsf{A} \to X$ make $X$ into a left and a right $\mathsf{A}$-module in $\mathcal{C}$, respectively. Moreover, the equation
\begin{equation*}
  \eval_X \circ (\id_{X^*} \otimes \uprho^{\ell}_X)
  = \eval_X \circ (\uprho^{r}_{X^*} \otimes \id_X)
\end{equation*}
holds for all $X \in \mathcal{C}$. Namely, the right $\mathsf{A}$-module $(X^*, \uprho_{X^*}^r)$ is the left dual of the left $\mathsf{A}$-module $(X, \uprho^{\ell}_X)$.

We recall that $\mathsf{A} = \mathsf{Z}(\unitobj)$ has the canonical half-braiding. Since the universal dinatural transformation $\uppi$ may not be easy to describe in practical cases, we provide the following formula expressing the canonical actions in terms of the canonical half-braiding.

\begin{lemma}
  Let $\sigma_X: \mathsf{A} \otimes X \to X \otimes \mathsf{A}$ be the canonical half-braiding for $\mathsf{A}$. Then the left and the right canonical action of $\mathsf{A}$ on $X$ are given respectively by
  \begin{equation}
    \uprho_X^{\ell} = (\id_X \otimes \upepsilon_{\unitobj}) \circ \sigma_X
    \quad \text{and} \quad
    \uprho_X^{r} = (\upepsilon_{\unitobj} \otimes \id_X) \circ \sigma_X^{-1}.
  \end{equation}
\end{lemma}
\begin{proof}
  The first equation follows from~\eqref{eq:Hom-V-ZW-br} with $t = \id_{\mathsf{Z}(\unitobj)}$. We consider the second one. For all objects $X \in \mathcal{C}$,
  \begin{align*}
    \uprho_X^{r} \circ \sigma_{X}
    & = (\eval'_{X} \otimes \id_X) \circ (\id_X \otimes \uppi_{\,{}^*\!X; \unitobj}) \\
    & \qquad \circ (\id_{X} \otimes \id_{\mathsf{Z}(\unitobj)} \otimes \eval_X)
      \circ (\uppi_{\mathsf{Z}(\unitobj); X} \otimes \id_X) \circ (\updelta_{\unitobj} \otimes \id_{X}) \\
    & = (\eval_X' \otimes \id_X \otimes \eval_{X}) \circ (\uppi_{\unitobj; X \otimes {}^* \! X} \otimes \id_X) \\
    & = (\id_X \otimes \eval_{X}) \circ ((\eval'_X)^* \otimes \id_X) \circ (\uppi_{\unitobj; \unitobj} \otimes \id_X) \\
    & = \upepsilon_{\unitobj} \otimes \id_X.
  \end{align*}
  Thus we have $\uprho_X^{r} = (\upepsilon_{\unitobj} \otimes \id_X) \circ \sigma_{X}^{-1}$.
\end{proof}

The coend $\mathsf{F} = \int^{X \in \mathcal{C}} X \otimes {}^* \! X$ has a canonical coalgebra structure. As explained in \cite{2015arXiv150401178S}, the adjoint algebra is the dual of the coalgebra $\mathsf{F}$. The category of $\mathsf{F}$-comodules in $\mathcal{C}$ has been studied by Lyubashenko \cite[Corollary 2.7.2]{MR1625495}. Translating his result into our context, we have the following description of the category of $\mathsf{A}$-modules in $\mathcal{C}$ (see \cite[Lemma 3.5]{2016arXiv160805905S} for an alternative proof).

\begin{lemma}
  \label{lem:adj-alg-mod}
  The following two functors are equivalences of categories\textup{:}
  \begin{align}
    \label{eq:adj-alg-mod-l}
    \mathcal{C} \boxtimes \mathcal{C} \to {}_\mathsf{A}\mathcal{C},
    \quad V \boxtimes W \mapsto (V \otimes W, \uprho^{\ell}_{V} \otimes \id_W), \\
    \label{eq:adj-alg-mod-r}
    \mathcal{C} \boxtimes \mathcal{C} \to \mathcal{C}_\mathsf{A},
    \quad V \boxtimes W \mapsto (V \otimes W, \id_V \otimes \uprho^r_{W}).
  \end{align}
\end{lemma}

\subsubsection*{Multiplicative characters}

Let $X \in \mathcal{C}$ be an object. Then the object $X \otimes X^*$ has a canonical structure of algebras in $\mathcal{C}$. It is easy to see that $\coev_X$ and $\uppi_{\unitobj; X}$ are morphisms of algebras in $\mathcal{C}$. Thus, if $\alpha \in \mathcal{C}$ is an invertible object, then
\begin{equation}
  \label{eq:def-chi-alpha}
  \upchi_{\alpha} := (\coev_{\alpha})^{-1} \circ \uppi_{\unitobj; \alpha}
  \in \Hom_{\mathcal{C}}(\mathsf{A}, \unitobj)
\end{equation}
is a morphism of algebras in $\mathcal{C}$. We call $\upchi_{\alpha}$ the {\em multiplicative character associated to the invertible object $\alpha \in \mathcal{C}$}.

Now let $\mathrm{Inv}(\mathcal{C})$ denote the set of isomorphism classes of invertible objects of $\mathcal{C}$. For two algebras $P$ and $Q$ in $\mathcal{C}$, we denote by $\mathrm{Alg}(P, Q)$ the set of algebra morphisms from $P$ to $Q$. Lemma~\ref{lem:mul-char} below has been proved in \cite{2016arXiv160805905S}. We include the proof for the sake of completeness.

\begin{lemma}
  \label{lem:mul-char}
  The following map is a well-defined bijection\textup{:}
  \begin{equation*}
    \upchi: \mathrm{Inv}(\mathcal{C}) \to \mathrm{Alg}(\mathsf{A}, \unitobj),
    \quad [\alpha] \mapsto \upchi_{\alpha}.
  \end{equation*}
\end{lemma}
\begin{proof}
  The well-definedness follows from the dinaturality of $\uppi$. To show the surjectivity, we remark that the set $\mathrm{Alg}(\mathsf{A}, \unitobj)$ is in bijection with the set of the isomorphism classes of left $\mathsf{A}$-modules in $\mathcal{C}$ whose underlying object is isomorphic to the unit object. Now let $(M, \rho)$ be such a left $\mathsf{A}$-module. Since $M \cong \unitobj$ in $\mathcal{C}$, $(M, \rho)$ is a simple object of ${}_{\mathsf{A}}\mathcal{C}$. Every simple object of $\mathcal{C} \boxtimes \mathcal{C}$ is of the form $V \boxtimes W$ for some simple objects $V, W \in \mathcal{C}$. Thus, by Lemma~\ref{lem:adj-alg-mod},
  \begin{equation*}
    (M, \rho) \cong (V \otimes W, \uprho_{V}^{\ell} \otimes \id_W)
  \end{equation*}
  for some simple objects $V, W \in \mathcal{C}$ such that $V \otimes W \cong \unitobj$ in $\mathcal{C}$. We may assume that $\alpha := V$ is an invertible object of $\mathcal{C}$ and $W \cong \alpha^*$. Then we have
  \begin{equation*}
    (M, \rho) \cong (\alpha \otimes \alpha^*, \uprho_{\alpha}^{\ell} \otimes \id_{\alpha^*}) \cong (\unitobj, \upchi_{\alpha})
  \end{equation*}
  as left $\mathsf{A}$-modules in $\mathcal{C}$. Therefore the map $[\alpha] \mapsto \upchi_{\alpha}$ is surjective. The injectivity of the map follows from the above argument and the fact that
  \begin{equation*}
    \text{$\alpha \boxtimes \alpha^* \cong \beta \boxtimes \beta^*$ in $\mathcal{C} \boxtimes \mathcal{C}$}
    \iff \text{$\alpha \cong \beta$ in $\mathcal{C}$}
  \end{equation*}
  for any two invertible objects $\alpha, \beta \in \mathcal{C}$.
\end{proof}

\subsubsection*{Bimodules over the adjoint algebra}

Using Lemma~\ref{lem:adj-alg-mod}, we also obtain the following description of the category ${}_{\mathsf{A}}\mathcal{C}_{\mathsf{A}}$ of $\mathsf{A}$-bimodules in $\mathcal{C}$.

\begin{lemma}
  \label{lem:adj-alg-bimod}
  ${}_{\mathsf{A}} \mathcal{C}_{\mathsf{A}} \approx \mathcal{C} \boxtimes \mathcal{C} \boxtimes \mathcal{C}$.
\end{lemma}
\begin{proof}
  Let $F: \mathcal{C} \boxtimes \mathcal{C} \to {}_{\mathsf{A}}\mathcal{C}$ and $G: \mathcal{C} \boxtimes \mathcal{C} \to \mathcal{C}_{\mathsf{A}}$ denote the equivalences \eqref{eq:adj-alg-mod-l} and~\eqref{eq:adj-alg-mod-r}, respectively. We note that ${}_{\mathsf{A}}\mathcal{C}$ is a right $\mathcal{C}$-module category such that the category of right $\mathsf{A}$-modules in ${}_{\mathsf{A}}\mathcal{C}$ is precisely the category ${}_{\mathsf{A}}\mathcal{C}_{\mathsf{A}}$. If we regard $\mathcal{C} \boxtimes \mathcal{C}$ as a right $\mathcal{C}$-module category by
  \begin{equation*}
    (V \boxtimes W) \otimes X = V \boxtimes (W \otimes X) \quad (V, W, X \in \mathcal{C}),
  \end{equation*}
  then $F$ is in fact an equivalence of right $\mathcal{C}$-module categories. Thus it induces an equivalence $F_{\mathsf{A}}: (\mathcal{C} \boxtimes \mathcal{C})_{\mathsf{A}} \to {}_{\mathsf{A}}\mathcal{C}_{\mathsf{A}}$ between the categories of right $\mathsf{A}$-modules. We use this observation to establish the desired equivalence:
  \begin{equation*}
    \mathcal{C} \boxtimes \mathcal{C} \boxtimes \mathcal{C}
    \xrightarrow[\quad \id \boxtimes G \quad]{\approx}
    \mathcal{C} \boxtimes \mathcal{C}_{\mathsf{A}}
    = (\mathcal{C} \boxtimes \mathcal{C})_{\mathsf{A}}
    \xrightarrow[\quad F_{\mathsf{A}} \quad]{\approx}
    {}_{\mathsf{A}}\mathcal{C}_{\mathsf{A}}.
    \qedhere
  \end{equation*}
\end{proof}

The proof of this lemma shows that the functor $\mathcal{C} \boxtimes \mathcal{C} \boxtimes \mathcal{C} \to {}_{\mathsf{A}}\mathcal{C}_{\mathsf{A}}$ given by
\begin{equation}
  \label{eq:adj-alg-bimod}
  V \boxtimes W \boxtimes X \mapsto
  (V \otimes W \otimes X, \uprho_{V}^{\ell} \otimes \id_W \otimes \id_X, \id_V \otimes \id_{W} \otimes \uprho_{X}^{r})
\end{equation}
is an equivalence of categories. It does not seem to be known an explicit description of a quasi-inverse of this functor. For later use, we compute the object of $\mathcal{C} \boxtimes \mathcal{C} \boxtimes \mathcal{C}$ corresponding to particular $\mathsf{A}$-bimodules.

\begin{lemma}
  \label{lem:adj-alg-bimod-1}
  Let $V \in \mathcal{C}$ be an object. Then the object $\mathsf{Z}(V)$ is an $\mathsf{A}$-bimodule by the monoidal structure of $\mathsf{Z}$. This bimodule corresponds to
  \begin{equation*}
    \int^{X \in \mathcal{C}} X \boxtimes V \boxtimes X^*
    \in \mathcal{C} \boxtimes \mathcal{C} \boxtimes \mathcal{C}
  \end{equation*}
  via the equivalence~\eqref{eq:adj-alg-bimod}.
\end{lemma}
\begin{proof}
  Let $\widetilde{\Psi}: \mathcal{C} \boxtimes \mathcal{C} \boxtimes \mathcal{C} \to {}_{\mathsf{A}}\mathcal{C}_{\mathsf{A}}$ denote the equivalence given by~\eqref{eq:adj-alg-bimod}. We fix an object $V \in \mathcal{C}$ and then define
  \begin{equation*}
    \Psi: \mathcal{C}^{\op} \times \mathcal{C}
    \to {}_{\mathsf{A}}\mathcal{C}_{\mathsf{A}},
    \quad (X_1, X_2) \mapsto \widetilde{\Psi}(X_2 \boxtimes V \boxtimes X_1^*).
  \end{equation*}
  It is easy to see that $q_X := \uppi_{V;X}: \mathsf{Z}(V) \to \Psi(X,X)$ is a morphism of $\mathsf{A}$-bimodules for all $X \in \mathcal{C}$. Thus $q = \{ q_X \}_{X \in \mathcal{C}}$ is a dinatural transformation $\mathsf{Z}(V) \dinto \Psi$. Now let $\omega: M \dinto \Psi$ be a dinatural transformation from $M \in {}_{\mathsf{A}}\mathcal{C}_{\mathsf{A}}$ to $\Psi$. Then, by the universal property of $\mathsf{Z}(V) \in \mathcal{C}$, there is a unique morphism $\psi: M \to \mathsf{Z}(V)$ in $\mathcal{C}$ such that $q_X \circ \psi = \omega_X$ for all $X \in \mathcal{C}$. The morphism $\psi$ belongs to ${}_{\mathsf{A}}\mathcal{C}_{\mathsf{A}}$. Indeed, if we denote by $\rho: \mathsf{A} \otimes M \to M$ the left action of $\mathsf{A}$ on $M$, then we have
  \begin{align*}
    \uppi_{V; X} \circ \psi \circ \rho
    & = \omega_X \circ \rho \\
    & = (\uprho_X^{\ell} \otimes \id_V \otimes \id_{X^*}) \circ \omega_X \\
    & = (\id_{X} \otimes \eval_X \otimes \id_{V} \otimes \id_{X^*})
      \circ (\uppi_{\unitobj; X} \otimes \uppi_{V; X})
      \circ (\id_{\mathsf{A}} \otimes \psi) \\
    & = \uppi_{V; X} \circ \mathsf{m}_{\unitobj, V} \circ (\id_{\mathsf{A}} \otimes \psi).
  \end{align*}
  This implies $\psi \circ \rho = \mathsf{m}_{\unitobj, V} \circ (\id_{\mathsf{A}} \otimes \psi)$, {\it i.e.}, $\psi$ is a morphism of left $\mathsf{A}$-modules. In a similar way, we can verify that $\psi$ is a morphism of right $\mathsf{A}$-modules.

  By the above argument, we conclude that $\mathsf{Z}(V) \in {}_{\mathsf{A}}\mathcal{C}_{\mathsf{A}}$ is an end of the functor $\Psi$. Since a category equivalence preserves ends, we have
  \begin{equation*}
    \widetilde{\Psi} \left(\int_{X \in \mathcal{C}} X \boxtimes V \boxtimes X^*\right)
    \cong \int_{X \in \mathcal{C}} \widetilde{\Psi}(X \boxtimes V \boxtimes X^*)
    \cong \int_{X \in \mathcal{C}} \Psi(X,X)
    \cong \mathsf{Z}(V)
  \end{equation*}
  as $\mathsf{A}$-bimodules. The claim is proved.
\end{proof}

In the proof of Lemma~\ref{lem:mul-char}, we have observed that the left $\mathsf{A}$-module $(\unitobj, \upchi_{\alpha})$ for an invertible object $\alpha \in \mathcal{C}$ corresponds to $\alpha \boxtimes \alpha^* \in \mathcal{C} \boxtimes \mathcal{C}$ via the equivalence given by \eqref{eq:adj-alg-mod-l}. The following lemma can be proved in a similar way:

\begin{lemma}
  \label{lem:adj-alg-bimod-2}
  Given invertible objects $\alpha, \beta \in \mathcal{C}$, we denote by ${}_{\alpha} \unitobj_{\beta}$ the $\mathsf{A}$-bimodule whose underlying object is $\unitobj$ with the left and the right action given by $\upchi_{\alpha}$ and $\upchi_{\beta}$, respectively. The $\mathsf{A}$-bimodule ${}_{\alpha}\unitobj_{\beta}$ corresponds to
  \begin{equation*}
    \alpha \boxtimes (\alpha^* \otimes \beta) \boxtimes \beta^*
    \in \mathcal{C} \boxtimes \mathcal{C} \boxtimes \mathcal{C}
  \end{equation*}
  via the equivalence~\eqref{eq:adj-alg-bimod}.
\end{lemma}

\subsection{Summary of notations}
\label{subsec:notations}

For reader's convenience, we append a summary of notations introduced in this section. Given a finite tensor category $\mathcal{C}$, we denote by $\mathcal{Z}(\mathcal{C})$ the Drinfeld center of $\mathcal{C}$. The forgetful functor $\mathsf{U}: \mathcal{Z}(\mathcal{C}) \to \mathcal{C}$ has a right adjoint functor, say $\mathsf{R}$. We denote by
\begin{equation*}
  \mathsf{m}_{V,W}: \mathsf{R}(V) \otimes \mathsf{R}(W) \to \mathsf{R}(V \otimes W)
  \quad \text{and} \quad
  \mathsf{u}: \unitobj \to \mathsf{R}(\unitobj)
  \quad (V, W \in \mathcal{C})
\end{equation*}
the monoidal structure of $\mathsf{R}$.

The functor $\mathsf{Z}: \mathcal{C} \to \mathcal{C}$ given by~\eqref{eq:Z-end} is a Hopf comonad on $\mathcal{C}$, which we call the central Hopf comonad (Definition~\ref{def:Hopf-comonad-Z}). The structure morphisms of $\mathsf{Z}$ are expressed in terms of the universal dinatural transformation
\begin{equation*}
  \uppi_{V; X}: \mathsf{Z}(V) \to X \otimes V \otimes X^*
  \quad (V, X \in \mathcal{C}).
\end{equation*}
The comonad structure of $\mathsf{Z}$ is denoted by
\begin{equation*}
  \updelta_{V}: \mathsf{Z}(V) \to \mathsf{Z}^2(V)
  \quad \text{and} \quad
  \upepsilon_{V}: \mathsf{Z}(V) \to V
  \quad (V \in \mathcal{C})
\end{equation*}
The category $\lcomod{\mathsf{Z}}$ of $\mathsf{Z}$-comodules can be identified with $\mathcal{Z}(\mathcal{C})$. If we identify them, $\mathsf{R}$ is identified with the free $\mathsf{Z}$-comodule functor \eqref{eq:free-Z-comod}. Since then $\mathsf{Z} = \mathsf{U R}$, the monoidal structure of $\mathsf{Z}$ are denoted by the same symbols as $\mathsf{R}$.

The object $\mathsf{A} := \mathsf{U R}(\unitobj)$ is an algebra in $\mathcal{C}$ with multiplication $m_{\unitobj, \unitobj}: \mathsf{A} \otimes \mathsf{A} \to \mathsf{A}$ and unit $\mathsf{u}: \unitobj \to \mathsf{A}$. We call $\mathsf{A}$ the {\em adjoint algebra} of $\mathcal{C}$ (Definition~\ref{def:adj-alg}). The algebra $\mathsf{A}$ acts on every object $X \in \mathcal{C}$ from both sides by the canonical actions
\begin{equation*}
  \uprho_X^{\ell}: \mathsf{A} \otimes X \to X
  \quad \text{and} \quad
  \uprho_X^{r}: X \otimes \mathsf{A} \to X
\end{equation*}
(Definition~\ref{def:can-act}). These notions have been used to establish several category equivalences; see Lemma~\ref{lem:adj-alg-mod} and~\ref{lem:adj-alg-bimod}.

\begin{example}[the case of Hopf algebras]
  \label{ex:adj-Hopf-alg}
  Let $H$ be a finite-dimensional Hopf algebra over $k$ with comultiplication $\Delta$, counit $\varepsilon$ and antipode $S$. We will use the Sweedler notation, such as $\Delta(h) = h_{(1)} \otimes h_{(2)}$ and
  \begin{equation*}
    \Delta(h_{(1)}) \otimes h_{(2)}
    = h_{(1)} \otimes h_{(2)} \otimes h_{(3)}
    = h_{(1)} \otimes \Delta(h_{(2)}),
  \end{equation*}
  to express the comultiplication of $h \in H$. We conclude this section by explaining what the above categorical notions indicate if $\mathcal{C} = \lmod{H}$ (see also \cite{2015arXiv150401178S}).

  We recall that a Yetter-Drinfeld $H$-module is a left $H$-module $M$ endowed with a left $H$-comodule structure $m \mapsto m_{(-1)} \otimes m_{(0)}$ satisfying
  \begin{equation*}
    (h \cdot m)_{(-1)} \otimes (h \cdot m)_{(0)}
    = h_{(1)} m_{(-1)} S(h_{(3)}) \otimes h_{(2)} m_{(0)}
  \end{equation*}
  for all $h \in H$ and $m \in M$. As is well-known, $\mathcal{Z}(\mathcal{C})$ for $\mathcal{C} = \lmod{H}$ can be identified with the category ${}^H_H \mathcal{YD}_f$ of finite-dimensional Yetter-Drinfeld $H$-modules. Under this identification, a right adjoint $\mathsf{R}$ of $\mathsf{U}$ is given as follows: As a vector space, it is given by $\mathsf{R}(V) = H \otimes_k V$ for $V \in \lmod{H}$. The left action $\triangleright$ and the left coaction of $H$ on the vector space $\mathsf{R}(V)$ are given respectively by
  \begin{equation}
    \label{eq:R-ind-action}
    h \triangleright (a \otimes v) = h_{(1)} a S(h_{(3)}) \otimes h_{(2)} v
    \quad \text{and} \quad
    a \otimes v \mapsto a_{(1)} \otimes a_{(2)} \otimes v
  \end{equation}
  for $a, h \in H$ and $v \in V$. The monoidal structure
  \begin{equation*}
    \mathsf{m}_{V,W}: \mathsf{R}(V) \otimes_k \mathsf{R}(W) \to \mathsf{R}(V \otimes_k W)
    \quad \text{and} \quad \mathsf{u}: k \to \mathsf{R}(k)
  \end{equation*}
  for $V, W \in \lmod{H}$ are given respectively by
  \begin{equation}
    \label{eq:R-ind-mult}
    \mathsf{m}_{V,W}((a \otimes v) \otimes (b \otimes w))
    = a b \otimes (v \otimes w)
    \quad \text{and} \quad
    \mathsf{u}(c) = 1_H \otimes c,
  \end{equation}
  for $a, b \in H$, $v \in V$, $w \in W$ and $c \in k$.

  The functor $\mathsf{Z} = \mathsf{U} \mathsf{R}$ has a structure of a Hopf comonad on $\mathcal{C}$. The comultiplication $\updelta_V: \mathsf{Z}(V) \to \mathsf{Z}^2(V)$ and the counit $\upepsilon_V: \mathsf{Z}(V) \to V$ of $\mathsf{Z}$ are determined by the unit and the counit of the adjunction $\mathsf{U} \dashv \mathsf{R}$. Explicitly, the natural transformations $\updelta$ and $\upepsilon$ are given respectively by
  \begin{align}
    \label{eq:R-ind-comult}
    \updelta_V(a \otimes v) = a_{(1)} \otimes a_{(2)} \otimes v
    \quad \text{and} \quad
    \upepsilon_V(a \otimes v) = \varepsilon(a) v,
  \end{align}
  for $a \in H$ and $v \in V \in \lmod{H}$. The monoidal structure of $\mathsf{Z}$ are the same natural transformations as $\mathsf{R}$.

  The adjoint algebra of $\lmod{H}$ is the vector space $\mathsf{A} := H$ with the adjoint action given by $h \triangleright a = h_{(1)} a S(h_{(2)})$ for $h \in H$ and $a \in \mathsf{A}$. The algebra structure of $\mathsf{A}$ is the same as the ordinary algebra $H$. The canonical half-braiding $\sigma_X: \mathsf{A} \otimes X \to X \otimes \mathsf{A}$ for $\mathsf{A}$ and its inverse are given respectively by
  \begin{equation*}
    \sigma_X(a \otimes x) = a_{(1)} x \otimes a_{(2)}
    \quad \text{and} \quad \sigma_X^{-1}(x \otimes a) = a_{(2)} \otimes S^{-1}(a_{(1)}) x
  \end{equation*}
  for $a \in \mathsf{A}$ and $x \in X \in \lmod{H}$. Thus, by Lemma~\ref{lem:adj-alg-mod}, we have
  \begin{equation}
    \uprho^{\ell}_X(a \otimes x) = a x
    \quad \text{and} \quad
    \uprho^r_X(x \otimes a) = S^{-1}(a) x.
  \end{equation}
  for $a \in \mathsf{A}$ and $x \in X \in \lmod{H}$.
\end{example}

\section{Integrals for finite tensor categories}
\label{sec:integ-ftc}

\subsection{Categorical analogue of the modular function}

Till the end of this section, we fix a finite tensor category $\mathcal{C}$ over $k$ and use the set of notations summarized in Subsection~\ref{subsec:notations}. The following lemma is a key observation for defining integrals and cointegrals of $\mathcal{C}$.

\begin{lemma}
  \label{lem:alpha-def}
  There is a unique \textup{(}up to isomorphism\textup{)} simple object $\upalpha \in \mathcal{C}$ such that
  \begin{equation}
    \label{eq:alpha-def}
    \Hom_{\mathcal{Z}(\mathcal{C})}(\mathsf{R}(\upalpha), \unitobj) \ne 0.
  \end{equation}
  The object $\upalpha$ is invertible. Moreover, there is a natural isomorphism
  \begin{equation}
    \label{eq:alpha-and-R}
    \Hom_{\mathcal{Z}(\mathcal{C})}(\mathsf{R}(X), \unitobj) \cong \Hom_{\mathcal{C}}(X, \upalpha)
    \quad (X \in \mathcal{C}).
  \end{equation}
\end{lemma}
\begin{proof}
  Let $D \in \mathcal{C}$ be the {\em distinguished invertible object} of $\mathcal{C}$ introduced by Etingof, Nikshych and Ostrik in \cite[Definition 3.1]{MR2097289}. Then the functor
  \begin{equation*}
    \mathsf{L}: \mathcal{C} \to \mathcal{Z}(\mathcal{C}),
    \quad X \mapsto \mathsf{R}(D^* \otimes X)
    \quad (X \in \mathcal{C})
  \end{equation*}
  is {\em left} adjoint to $U$ \cite[Lemma 4.7]{2014arXiv1402.3482S}. We recall also from \cite{MR2097289} that, as its name suggests, the object $D$ is an invertible object. Thus we have isomorphisms
  \begin{align*}
      \Hom_{\mathcal{Z}(\mathcal{C})}(\mathsf{R}(X), \unitobj)
      & \cong \Hom_{\mathcal{Z}(\mathcal{C})}(\mathsf{R}(D^* \otimes D \otimes X), \unitobj) \\
      & \cong \Hom_{\mathcal{Z}(\mathcal{C})}(\mathsf{L}(D \otimes X), \unitobj) \\
      & \cong \Hom_{\mathcal{Z}(\mathcal{C})}(D \otimes X, U(\unitobj)) \\
      & \cong \Hom_{\mathcal{Z}(\mathcal{C})}(X, D^*)
  \end{align*}
  natural in the variable $X \in \mathcal{C}$. Now suppose that $X \in \mathcal{C}$ is a simple object. Then, by the above computation and Schur's lemma, we have
  \begin{equation*}
    \Hom_{\mathcal{Z}(\mathcal{C})}(\mathsf{R}(X), \unitobj) \ne 0
    \iff X \cong D^*.
  \end{equation*}
  Thus $\upalpha := D^*$ is a unique (up to isomorphism) simple object satisfying~\eqref{eq:alpha-def}. The invertibility of $\upalpha$ follows from the invertibility of $D$. The isomorphism~\eqref{eq:alpha-and-R} has been given in this proof.
\end{proof}

\begin{definition}
  \label{def:ftc-mod-ft}
  We call the object $\upalpha$ of Lemma~\ref{lem:alpha-def} the {\em modular object} of $\mathcal{C}$. We say that $\mathcal{C}$ is {\em unimodular} if $\upalpha$ is isomorphic to the unit object.
\end{definition}

As we have seen in the proof of Lemma~\ref{lem:alpha-def}, the object $\upalpha$ is in fact the dual of the distinguished invertible object introduced in \cite{MR2097289}. Thus our definition of the unimodularity agrees with that given in \cite{MR2097289}. One can find many other properties of $\upalpha$ in, for example, \cite{MR2097289} and \cite{2014arXiv1402.3482S}.

\subsection{Definition of integrals and cointegrals}
\label{subsec:integ-FTCs-def}

We now introduce the notions of integrals and cointegrals of the finite tensor category $\mathcal{C}$.

\begin{definition}
  \label{def:integ-FTCs}
  Let $\upalpha$ be the modular object of $\mathcal{C}$. A {\em categorical cointegral} of $\mathcal{C}$ is a morphism $\lambda: \mathsf{Z}(\upalpha) \to \unitobj$ in $\mathcal{C}$ satisfying
  \begin{equation}
    \label{eq:cointegral-def}
    \mathsf{Z}(\lambda) \circ \updelta_{\upalpha} = \mathsf{u} \circ \lambda.
  \end{equation}
  A {\em categorical left integral} of $\mathcal{C}$ is a morphism $\Lambda: \unitobj \to \mathsf{Z}(\upalpha)$ in $\mathcal{C}$ satisfying
  \begin{equation}
    \label{eq:integral-def-1}
    \mathsf{m}_{\unitobj,\upalpha} \circ (\id_{\mathsf{Z}(\unitobj)} \otimes \Lambda) = \upepsilon_{\unitobj} \otimes \Lambda.
  \end{equation}
  A {\em categorical right integral} of $\mathcal{C}$ is a morphism $\Lambda: \unitobj \to \mathsf{Z}(\upalpha)$ in $\mathcal{C}$ satisfying
  \begin{equation}
    \label{eq:integral-def-2}
    \mathsf{m}_{\upalpha, \unitobj} \circ (\Lambda \otimes \id_{\mathsf{Z}(\unitobj)}) = \Lambda \otimes \upepsilon_{\unitobj}.
  \end{equation}
\end{definition}

The adjective `categorical' is omitted when no confusion arises. A cointegral of a finite tensor category $\mathcal{C}$ may be referred to as a categorical `left' cointegral of $\mathcal{C}$ in view of Example~\ref{ex:Hopf-integ} below. One may define a `right' cointegral, which is not needed in this paper, as a left cointegral of $\mathcal{C}^{\rev}$.

If $\mathcal{C}$ is unimodular, {\it i.e.}, $\upalpha \cong \unitobj$, then our Definition~\ref{def:integ-FTCs} agrees with the definitions given in \cite{2015arXiv150401178S}. As we will see in Theorem~\ref{thm:integ-exist} and its corollary, left integrals and right integrals are the same notion in the unimodular case. 

\begin{example}
  \label{ex:Hopf-integ}
  We justify the above terminologies by considering the case where $\mathcal{C}$ is the representation category of a Hopf algebra. Let $H$ be a finite-dimensional Hopf algebra over $k$. We recall that a {\em left integral} of $H$ is an element $\Lambda \in H$ such that
  \begin{equation}
    \label{eq:Hopf-integ-def-1}
    \text{$h \Lambda = \upepsilon(h) \Lambda$ for all $h \in H$}.  
  \end{equation}
  A {\em left cointegral} of $H$ is a left integral of the dual Hopf algebra of $H$. Namely, a left cointegral of $H$ is a linear map $\lambda: H \to k$ such that
  \begin{equation}
    \label{eq:Hopf-integ-def-2}
    \text{$a_{(1)} \lambda(a_{(2)}) = \lambda(a) 1_H$ for all $a \in H$}.  
  \end{equation}
  It is known that a left integral of $H$ always exists and is unique up to scalar multiple. The {\em modular function} (also called the distinguished grouplike element of $H^*$ in \cite{MR1265853}) is defined to be the unique linear map $\upalpha: H \to k$ such that
  \begin{equation}
    \label{eq:Hopf-alpha-def}
    \text{$\Lambda h = \upalpha(h) \Lambda$ for all $h \in H$},
  \end{equation}
  where $\Lambda \in H$ is an arbitrary non-zero left integral ({\it cf}. the definition of the modular function on a locally compact group). The modular function $\upalpha: H \to k$ is an algebra map. By abuse of notation, we denote by the same symbol $\upalpha$ the left $H$-module corresponding to the algebra map $\upalpha: H \to k$. The object $\upalpha \in \lmod{H}$ is in fact the modular object of $\lmod{H}$, as equation \eqref{eq:cat-integ-Hopf-2} below shows.

  As in Example~\ref{ex:adj-Hopf-alg}, we identify $\mathcal{Z}(\lmod{H})$ with ${}^H_H \mathcal{YD}_{f}$. We examine categorical cointegrals of $\lmod{H}$ by using the description of $\mathsf{R}$ given in Example~\ref{ex:adj-Hopf-alg}. By~\eqref{eq:R-ind-mult}, a linear map $\lambda: \mathsf{Z}(\upalpha) \to k$ is $H$-linear if and only if
  \begin{equation}
    \label{eq:cat-integ-Hopf-1}
    \text{$\lambda(h_{(1)} a S(h_{(3)})) \upalpha(h_{(2)}) = \varepsilon(a) \lambda(a) 1_H$ for all $a, h \in H$}.
  \end{equation}
  By \eqref{eq:R-ind-comult}, the linear map $\lambda$ satisfies~\eqref{eq:cointegral-def} if and only if \eqref{eq:Hopf-integ-def-2} holds, that is, $\lambda$ is a left cointegral of $H$. By the results of Radford \cite{MR1265853}, equation~\eqref{eq:cat-integ-Hopf-1} is automatically satisfied if $\lambda$ is a left cointegral of $H$. In summary, we have verified that a linear map $\lambda: \mathsf{Z}(\upalpha) \to k$ is a categorical cointegral of $\lmod{H}$ if and only if it is a left cointegral of $H$. We note that~\eqref{eq:cointegral-def} is also equivalent to that the map $\lambda: \mathsf{R}(\upalpha) \to k$ is $H$-colinear. Thus, we also have
  \begin{equation}
    \label{eq:cat-integ-Hopf-2}
    \Hom_{\mathcal{Z}(\lmod{H})}(\mathsf{R}(\upalpha), k)
    = \{ \lambda \in H^* \mid \text{$\lambda$ is a left cointegral of $H$} \} \ne 0,
  \end{equation}
  which means that $\upalpha$ is indeed the modular object of $\lmod{H}$.

  We now consider categorical left integrals of $\lmod{H}$. In what follows, we identify the vector space $\mathsf{Z}(\upalpha) = H \otimes_k \upalpha$ with $H$. Let $t: k \to \mathsf{Z}(\upalpha)$ be a linear map, and set $\Lambda = t(1) \in H$. Then, by~\eqref{eq:R-ind-action}, the $H$-linearity of $t$ is equivalent to
  \begin{equation}
    \label{eq:cat-integ-Hopf-3}
    \text{$h_{(1)} \Lambda S(h_{(3)}) \upalpha(h_{(2)}) = \varepsilon(h) \Lambda$ for all $h \in H$}.
  \end{equation}
  By~\eqref{eq:R-ind-mult}, the linear map $t: k \to \mathsf{Z}(\upalpha)$ satisfies \eqref{eq:integral-def-1} if and only if \eqref{eq:Hopf-integ-def-1} holds, that is, the element $\Lambda \in H$ is a left integral. By \eqref{eq:Hopf-alpha-def}, equation \eqref{eq:cat-integ-Hopf-3} is automatically satisfied when $\Lambda$ is a left integral of $H$. In summary, the map
  \begin{gather*}
    \{ t \in \Hom_H(k, \mathsf{Z}(\upalpha)) \mid \text{$t$ is a categorical left integral of $\lmod{H}$} \} \\
    \longrightarrow \{ \Lambda \in H \mid \text{$\Lambda$ is a left integral of $H$} \},
    \qquad t \mapsto t(1)
  \end{gather*}
  is a bijection. Similarly, a categorical right integral of $\lmod{H}$ can be identified with a right integral of $H$ in a natural way.
\end{example}

\subsection{Frobenius type property of the adjoint algebra}

It is a fundamental result that a non-zero left integral of a finite-dimensional Hopf algebra always exists and unique up to scalar multiples. Our first aim is to establish a categorical version of this fact. For this purpose, we introduce some morphisms that may be considered as a `Frobenius structure' of the adjoint algebra of $\mathcal{C}$. More precisely, we consider the algebra $\mathsf{R}(\unitobj)$ in $\mathcal{Z}(\mathcal{C})$. By abuse of notation, we write $\mathsf{A} = \mathsf{R}(\unitobj)$.

The algebra $\mathsf{A} \in \mathcal{Z}(\mathcal{C})$ acts on every object of the form $\mathsf{R}(V)$, $V \in \mathcal{C}$, by the monoidal structure of $\mathsf{R}$. To clarify, we denote $\mathsf{R}(V)$ by $\mathsf{R}(V)_{\mathsf{A}}$ when it is viewed as a right $\mathsf{A}$-module in this way. There is a $k$-linear functor
\begin{equation}
  \label{eq:FTHM-ZC-1}
  \mathcal{C} \to \mathcal{Z}(\mathcal{C})_{\mathsf{A}},
  \quad V \mapsto \mathsf{R}(V)_{\mathsf{A}}.
\end{equation}
One can prove that this functor is an equivalence of categories by considering the internal Hom functor for the $\mathcal{Z}(\mathcal{C})$-module category $\mathcal{C}$ \cite[Theorem 5.6]{2014arXiv1402.3482S}. Another proof is given by applying the fundamental theorem for Hopf modules to the Hopf comonad $\mathsf{Z}$ \cite{2016arXiv160805905S}.

Similarly, we deote by ${}_{\mathsf{A}}\mathsf{R}(V)$ the object $\mathsf{R}(V) \in \mathcal{Z}(\mathcal{C})$ regarded as a left $\mathsf{A}$-module by the monoidal structure of $\mathsf{R}$. There is an equivalence
\begin{equation}
  \label{eq:FTHM-ZC-2}
  \mathcal{C} \to \mathcal{Z}(\mathcal{C})_{\mathsf{A}},
  \quad V \mapsto {}_{\mathsf{A}} \mathsf{R}(V)
\end{equation}
of $k$-linear categories. Now we use these equivalences to prove:

\begin{lemma}
  $\mathsf{A}^* \cong \mathsf{R}(\upalpha)_{\mathsf{A}}$ as right $\mathsf{A}$-modules in $\mathcal{Z}(\mathcal{C})$.
\end{lemma}
\begin{proof}
  The proof of this lemma is given in \cite{2014arXiv1402.3482S}. We give a self-contained proof for the completeness: The left $\mathsf{A}$-module $\mathsf{A}$ is a simple object of ${}_{\mathsf{A}}\mathcal{Z}(\mathcal{C})$. Since the duality functor $(-)^*: {}_{\mathsf{A}}\mathcal{Z}(\mathcal{C}) \to \mathcal{Z}(\mathcal{C})_{\mathsf{A}}$ is an anti-equivalence, $\mathsf{A}^*$ is a simple object of $\mathcal{Z}(\mathcal{C})_{\mathsf{A}}$. Thus, in view of the equivalence~\eqref{eq:FTHM-ZC-1}, there is a simple object $V \in \mathcal{C}$ such that $\mathsf{A}^* \cong \mathsf{R}(V)$ as right $\mathsf{A}$-modules. By Lemmas~\ref{lem:module-adj} and \ref{lem:alpha-def},
  \begin{gather*}
    \Hom_{\mathcal{C}}(V, \upalpha)
    \cong \Hom_{\mathcal{Z}(\mathcal{C})}(\mathsf{R}(V), \unitobj)
    \cong \Hom_{-|\mathsf{A}}(\mathsf{R}(V), \mathsf{A}^*) \ne 0.
  \end{gather*}
  By Schur's lemma, $V \cong \upalpha$. Thus $\mathsf{A}^* \cong \mathsf{R}(V)$ as right $\mathsf{A}$-modules in $\mathcal{Z}(\mathcal{C})$.
\end{proof}

We fix an isomorphism $\phi: \mathsf{R}(\upalpha) \to \mathsf{A}^*$ of right $\mathsf{A}$-modules in $\mathcal{Z}(\mathcal{C})$ and set
\begin{equation}
  \label{eq:adj-alg-Fb-1}
  e_{\phi} = \eval_{\mathsf{A}} \circ (\phi \otimes \id_{\mathsf{A}})
  \quad \text{and} \quad
  d_{\phi} = (\id_{\mathsf{A}} \otimes \phi^{-1}) \circ \coev_{\mathsf{A}}.
\end{equation}
Then the triple $(\mathsf{R}(\upalpha), e_{\phi}, d_{\phi})$ is a left dual object of $\mathsf{A}$. Since $\phi$ is $\mathsf{A}$-linear, the evaluation morphism $e_{\phi}$ is a `Frobenius form' in the sense that the equation
\begin{equation}
  \label{eq:adj-alg-Fb-2}
  e_{\phi} \circ (\mathsf{m}_{\unitobj, \upalpha} \otimes \id_{\mathsf{A}}) = e_{\phi} \circ (\id_{\mathsf{R}(\upalpha)} \otimes \mathsf{m}_{\unitobj,\unitobj})
\end{equation}
holds. Similarly, the coevaluation morphism $d_{\phi}$ satisfies
\begin{equation}
  \label{eq:adj-alg-Fb-3}
  (\id_{\mathsf{A}} \otimes \mathsf{m}_{\unitobj, \upalpha}) \circ (d_{\phi} \otimes \id_{\mathsf{A}})
  = (\mathsf{m}_{\unitobj,\unitobj} \otimes \id_{\mathsf{R}(\upalpha)}) \circ (\id_{\mathsf{A}} \otimes d_{\phi}).
\end{equation}
We define the `Frobenius trace' of $\mathsf{A}$ associated to $\phi$ by
\begin{equation}
  \label{eq:adj-alg-Fb-4}
  t_{\phi} = \mathsf{u}^* \circ \phi.
\end{equation}
Then, as in the ordinary theory of Frobenius algebras, we have
\begin{equation}
  \label{eq:adj-alg-Fb-5}
  e_{\phi} = t_{\phi} \circ \mathsf{m}_{\upalpha, \unitobj}.
\end{equation}
Indeed, by~\eqref{eq:adj-alg-Fb-1}--\eqref{eq:adj-alg-Fb-4} and the definition of the dual morphism,
\begin{align*}
  e_{\phi}
  & = e_{\phi} \circ (\id_{\mathsf{R}(\upalpha)} \otimes \mathsf{m}_{\unitobj,\unitobj}) \circ (\id_{\mathsf{R}(\upalpha)} \otimes \id_{\mathsf{A}} \otimes \mathsf{u}) \\
  & = e_{\phi} \circ (\mathsf{m}_{\upalpha,\unitobj} \otimes \id_{\mathsf{A}}) \circ (\id_{\mathsf{R}(\upalpha)} \otimes \id_{\mathsf{A}} \otimes \mathsf{u}) \\
  & = \eval_{\mathsf{A}} \circ (\phi \otimes \id_{\mathsf{A}}) \circ (\mathsf{m}_{\upalpha, \unitobj} \otimes \mathsf{u}) \\
  & = \eval_{\unitobj} \circ (\mathsf{u}^* \otimes \id_{\unitobj}) \circ (\phi \otimes \id_{\unitobj}) \circ (\mathsf{m}_{\upalpha, \unitobj} \otimes \id_{\unitobj}) \\
  & = \mathsf{u}^* \circ \phi \circ \mathsf{m}_{\upalpha, \unitobj}
    = t_{\phi} \circ \mathsf{m}_{\upalpha, \unitobj}.
\end{align*}

\begin{lemma}
  \label{lem:coint-and-trace}
  With the notation as above,
  \begin{equation}
    \label{eq:coint-and-trace}
    \Hom_{\mathcal{Z}(\mathcal{C})}(\mathsf{R}(\upalpha), \unitobj) = \mathrm{span}_k \{ t_{\phi} \} \ne 0.
  \end{equation}
\end{lemma}
\begin{proof}
  Let $\mathbf{I} := \Hom_{\mathcal{Z}(\mathcal{C})}(\mathsf{R}(\upalpha), \unitobj)$ be the left-hand side of \eqref{eq:coint-and-trace}. Then $t_{\phi} \in \mathbf{I}$ since it is defined as the composition of morphisms in $\mathcal{Z}(\mathcal{C})$ as in \eqref{eq:adj-alg-Fb-4}. If $t_{\phi} = 0$, then $e_{\phi} = 0$ by~\eqref{eq:adj-alg-Fb-5}, and hence we have $\id_{\mathsf{A}} = (\id_{\mathsf{A}} \otimes e_{\phi}) \circ (d_{\phi} \otimes \id_{\mathsf{A}}) = 0$, a contradiction. Thus $t_{\phi} \ne 0$. By Lemma~\ref{lem:alpha-def} and Schur's lemma, we have
  \begin{equation*}
    \mathbf{I}
    \cong \Hom_{\mathcal{C}}(\upalpha, \upalpha) \cong k.
  \end{equation*}
  Therefore the vector space $\mathbf{I}$ is spanned by $t_{\phi}$ over $k$.
\end{proof}

\begin{remark}
  \label{rem:right-dual-A}
  In the same way as above, one can prove that there is also an isomorphism ${}^* \mathsf{A} \cong \mathsf{R}(\upalpha)$ of left $\mathsf{A}$-modules. We fix an isomorphism $\psi: \mathsf{R}(\upalpha) \to {}^* \mathsf{A}$ of left $\mathsf{A}$-modules in $\mathcal{Z}(\mathcal{C})$ and set
  \begin{equation*}
    e'_{\psi} = \eval'_{\mathsf{A}} \circ (\id_{\mathsf{A}} \otimes \psi)
    \quad \text{and} \quad
    d'_{\psi} = (\psi^{-1} \otimes \id_{\mathsf{A}}) \circ \coev'_{\mathsf{A}}.
  \end{equation*}
  The $\mathsf{A}$-linearity of $\psi$ implies the following equations:
  \begin{gather*}
    e'_{\psi} \circ (\id_\mathsf{A} \otimes \mathsf{m}_{\unitobj, \upalpha})
    = e'_{\psi} \circ (\mathsf{m}_{\unitobj, \unitobj} \otimes \id_{\mathsf{R}(\upalpha)}). \\
    (\mathsf{m}_{\unitobj, \upalpha} \otimes \id_{\mathsf{A}}) \circ (\id_{\mathsf{A}} \otimes d'_{\psi})
    = (\id_{\mathsf{R}(\upalpha)} \otimes \mathsf{m}_{\unitobj, \unitobj}) \circ (d'_{\psi} \otimes \id_{\mathsf{A}}).
  \end{gather*}
\end{remark}

\subsection{Existence and uniqueness}

The adjoint algebra $\mathsf{A} = \mathsf{U} \mathsf{R}(\unitobj)$ acts on every object of the form $\mathsf{Z}(V)$ by the monoidal structure of $\mathsf{Z}$. By the discussion of the previous subsection, we obtain an isomorphism $\mathsf{Z}(\upalpha) \cong \mathsf{A}^*$ of right $\mathsf{A}$-modules in $\mathcal{C}$ and an isomorphism ${}^*\mathsf{A} \cong \mathsf{Z}(\upalpha)$ of left $\mathsf{A}$-modules in $\mathcal{C}$. Now we are ready to prove:

\begin{theorem}
  \label{thm:integ-exist}
  Let $\mathbf{I}^{c}$, $\mathbf{I}^{\ell}$ and $\mathbf{I}^{r}$ be the spaces of cointegrals, left integrals and right integrals of the finite tensor category $\mathcal{C}$, respectively. Then we have
  \begin{equation}
    \label{eq:integ-exist-1}
    \dim_k \mathbf{I}^{c} = \dim_k \mathbf{I}^{\ell} = \dim_k \mathbf{I}^{r} = 1.
  \end{equation}
  For every non-zero cointegral $\lambda \in \mathbf{I}^c$, there exist a unique left integral $\Lambda^{\ell} \in \mathbf{I}^{\ell}$ and a unique right integral $\Lambda^r \in \mathbf{I}^{r}$ such that
  \begin{equation}
    \label{eq:integ-exist-2}
    \lambda \circ \Lambda^{\ell} = \id_{\unitobj} = \lambda \circ \Lambda^r.
  \end{equation}
\end{theorem}
\begin{proof}
  Let $\lambda: \mathsf{Z}(\upalpha) \to \unitobj$ be a morphism in $\mathcal{C}$. The condition~\eqref{eq:cointegral-def} is equivalent to that $\lambda$ is a morphism of $\mathsf{Z}$-comodules. Hence,
  \begin{equation*}
    \mathbf{I}^{c} = \Hom_{\lcomod{\mathsf{Z}}}(\mathsf{Z}(\upalpha), \unitobj)
    = \Hom_{\mathcal{Z}(\mathcal{C})}(\mathsf{R}(\upalpha), \unitobj).
  \end{equation*}
  Thus $\dim_k \mathbf{I}^c = 1$ by Lemma~\ref{lem:coint-and-trace}. We consider the space $\mathbf{I}^{r}$ of right integrals. If we view $\unitobj$ as a right $\mathsf{A}$-module by $\upepsilon_{\unitobj}: \mathsf{A} \to \unitobj$, then a right integral of $\mathcal{C}$ is just a morphism $\unitobj \to \mathsf{Z}(\upalpha)$ of right $\mathsf{A}$-modules. Thus, by Lemma~\ref{lem:module-adj},
  \begin{equation*}
    \mathbf{I}^{r} = \Hom_{-|\mathsf{A}}(\unitobj, \mathsf{Z}(\upalpha))
    \cong \Hom_{-|\mathsf{A}}(\unitobj, \mathsf{A}^*)
    \cong \Hom_{\mathcal{C}}(\unitobj, \unitobj) \cong k,
  \end{equation*}
  where $\Hom_{-|\mathsf{A}}$ is the Hom functor of $\mathcal{C}_{\mathsf{A}}$. Similarly, if we view $\unitobj$ as a left $\mathsf{A}$-module by $\upepsilon_{\unitobj}$, then we have
  \begin{equation*}
    \mathbf{I}^{\ell} = \Hom_{\mathsf{A}|-}(\unitobj, \mathsf{Z}(\upalpha))
    \cong \Hom_{\mathsf{A}|-}(\unitobj, {}^* \mathsf{A})
    \cong \Hom_{\mathcal{C}}(\unitobj, \unitobj) \cong k,
  \end{equation*}
  where $\Hom_{\mathsf{A}|-}$ is the Hom functor of ${}_{\mathsf{A}}\mathcal{C}$. Hence \eqref{eq:integ-exist-1} has been verified.

  Now let $\lambda \in \mathbf{I}^c$ be a non-zero cointegral of $\mathcal{C}$. We choose an isomorphism $\phi: \mathsf{A}^* \to \mathrm{R}(\upalpha)$ of right $\mathsf{A}$-modules in $\mathcal{Z}(\mathcal{C})$ and set $d_{\phi}$ and $t_{\phi}$ by \eqref{eq:adj-alg-Fb-1} and~\eqref{eq:adj-alg-Fb-4}, respectively. Lemma~\ref{lem:coint-and-trace} implies that $\lambda = c t_{\phi}$ for some $c \in k^{\times}$. If we set
  \begin{equation*}
    \Lambda_0 := (\upepsilon_{\unitobj} \otimes \id_{\mathsf{Z}(\upalpha)}) \circ d_{\phi},
  \end{equation*}
  then we have
  \begin{align*}
    m_{\upalpha, \unitobj} \circ (\Lambda_0 \otimes \id_a)
    & = (\upepsilon_{\unitobj} \otimes \id_{\mathsf{Z}(\upalpha)}) \circ (\id_{\mathsf{A}} \otimes m_{\upalpha, \unitobj}) \circ (d_{\phi} \otimes \id_\mathsf{A}) \\
    & = (\upepsilon_{\unitobj} \otimes \id_{\mathsf{Z}(\upalpha)}) \circ (m_{\unitobj,\unitobj} \otimes \id_{\mathsf{Z}(\upalpha)}) \circ (\id_{\mathbf{\mathsf{A}}} \otimes d_{\phi}) \\
    & = (\upepsilon_{\unitobj} \otimes \upepsilon_{\unitobj} \otimes \id_{\mathsf{Z}(\upalpha)}) \circ (\id_{\mathbf{\mathsf{A}}} \otimes d_{\phi}) \\
    & = \upepsilon_{\unitobj} \otimes \Lambda_0
  \end{align*}
by~\eqref{eq:adj-alg-Fb-3} and the fact that $\upepsilon_{\unitobj}: \mathsf{A} \to \unitobj$ is a morphism of algebras in $\mathcal{C}$. Thus $\Lambda_0$ is a right integral of $\mathcal{C}$. Moreover, we have
  \begin{align*}
    t_{\phi} \circ \Lambda_0
    & = (\upepsilon_{\unitobj} \otimes t_{\phi}) \circ d_{\phi}
      = (\upepsilon_{\unitobj} \otimes u^*) \circ \coev_{\mathsf{A}}
      = \upepsilon_{\unitobj} \circ u = \id_{\unitobj}.
  \end{align*}
  Thus $\Lambda^r := c^{-1} \Lambda_0$ is a right integral satisfying~\eqref{eq:integ-exist-2}. Since $\dim_k \mathbf{I}^r = 1$, such a right integral is unique. One can prove the existence and the uniqueness of a left integral $\Lambda^{\ell}$ satisfying~\eqref{eq:integ-exist-2} in a similar way (see Remark~\ref{rem:right-dual-A}).
\end{proof}

\subsection{Relation to the original definition of the modular function}

For a finite-dimensional Hopf algebra, the modular function is defined in terms of integrals. On the other hand, we have first defined the modular object of $\mathcal{C}$ and introduced categorical integrals by using the modular object. We recall that a morphism $\upchi_{\beta}: A \to \unitobj$ of algebras in $\mathcal{C}$ is defined by~\eqref{eq:def-chi-alpha} for each invertible object $\beta \in \mathcal{C}$. Motivated by the defining formula \eqref{eq:Hopf-alpha-def} of the modular function, we will prove:

\begin{theorem}
  \label{thm:integ-mod-ft}
  If $\Lambda^{\ell}$ is a left integral of $\mathcal{C}$, then we have
  \begin{equation}
    \label{eq:integral-alpha-1}
    \mathsf{m}_{\upalpha, \unitobj} \circ (\Lambda^{\ell} \otimes \id_A) = \Lambda^{\ell} \otimes \upchi_{\upalpha}.
  \end{equation}
  If $\Lambda^{r}$ is a right integral of $\mathcal{C}$, then we have
  \begin{equation}
    \label{eq:integral-alpha-2}
    \mathsf{m}_{\unitobj, \upalpha} \circ (\id_{A} \otimes \Lambda^r) = \upchi_{\upalpha^*} \otimes \Lambda^r.
  \end{equation}
\end{theorem}

The above theorem yields the following consequence:

\begin{corollary}
  \label{cor:unimo-integ}
  Let $\mathbf{I}^{\ell}$ and $\mathbf{I}^r$ be the spaces of left and right integrals of $\mathcal{C}$, respectively. Then $\mathbf{I}^{\ell} = \mathbf{I}^r$ if and only if $\mathcal{C}$ is unimodular.
\end{corollary}
\begin{proof}
  The `if' part is obvious from $\upepsilon_{\unitobj} = \upchi_{\unitobj}$ and the above theorem. To prove the converse, we assume $\mathbf{I}^{\ell} = \mathbf{I}^{r}$. Let $\Lambda$ be a non-zero left integral. Then it is also a right integral by the assumption. By the above theorem, we have
  \begin{equation*}
    \Lambda \circ \upchi_{\upalpha}
    = \Lambda \otimes \upchi_{\upalpha}
    = \mathsf{m}_{\upalpha, \unitobj} \circ (\Lambda \otimes \id_A)
    = \varepsilon_{\unitobj} \otimes \Lambda
    = \Lambda \circ \varepsilon_{\unitobj}.
  \end{equation*}
  The morphism $\Lambda$ is a monomorphism as a non-zero morphism from a simple object. Thus, by Lemma~\ref{lem:mul-char}, we have $\upalpha \cong \unitobj$, {\it i.e.}, $\mathcal{C}$ is unimodular.
\end{proof}

To prove Theorem~\ref{thm:integ-mod-ft}, we introduce the following notion: A {\em generalized right integral} is a pair $(V, t)$ consisting of an object $V \in \mathcal{C}$ and a morphism $t: V \to \mathsf{Z}(\upalpha)$ in $\mathcal{C}$ satisfying the following equation:
\begin{equation}
  \label{eq:def-gen-integ}
  \mathsf{m}_{\upalpha, \unitobj} \circ (t \otimes \id_A) = t \otimes \upepsilon_{\unitobj}.
\end{equation}
We define the category $\mathcal{I}^{r}$ of generalized right integrals as the full subcategory of the category of objects over $\mathsf{Z}(\upalpha) \in \mathcal{C}$. A categorical right integral has the following universal property:

\begin{lemma}
  \label{lem:gen-integ-terminal}
  If $\Lambda^{r}$ is a non-zero categorical right integral, then the pair $(\unitobj, \Lambda^{r})$ is a terminal object of the category $\mathcal{I}^{r}$.
\end{lemma}
\begin{proof}
  Given an object $V \in \mathcal{C}$, we denote by $V_{\upepsilon} \in \mathcal{C}_\mathsf{A}$ the object $V$ regarded as a right $\mathsf{A}$-module by the action $\id_V \otimes \upepsilon_{\unitobj}$. We fix an isomorphism $\phi: \mathsf{Z}(\upalpha) \to \mathsf{A}^*$ of right $\mathsf{A}$-modules in $\mathcal{C}$. Then we have natural isomorphisms
  \begin{equation}
    \label{eq:gen-integ-1}
    \Hom_{-|\mathsf{A}}(V_{\upepsilon}, \mathsf{Z}(\upalpha))
    \cong \Hom_{-|\mathsf{A}}(V_{\upepsilon}, \mathsf{A}^*)
    \cong \Hom_{\mathcal{C}}(V, \unitobj)
    \quad (V \in \mathcal{C}).
  \end{equation}
  Now let $(V, t)$ be a generalized left integral of $\mathcal{C}$. Equation~\eqref{eq:def-gen-integ} means that $t: V_{\upepsilon} \to \mathsf{Z}(\upalpha)$ is a morphism of right $\mathsf{A}$-modules. Let $t_0: V \to \unitobj$ be the morphism corresponding to $t$ via \eqref{eq:gen-integ-1}. We consider the commutative diagram
  \begin{equation*}
    \xymatrix{
      \Hom_{-|\mathsf{A}}(\unitobj_{\upepsilon}, \mathsf{Z}(\upalpha))
      \ar[d]_{\Hom_{-|\mathsf{A}}(t_0, \mathsf{Z}(\upalpha))}
      \ar[rrrr]^{\text{\eqref{eq:gen-integ-1} with $V = \unitobj$}}_{\cong}
      & & & & \Hom_{\mathcal{C}}(\unitobj, \unitobj)
      \ar[d]^{\Hom_{\mathcal{C}}(t_0, \unitobj)} \\
      \Hom_{-|\mathsf{A}}(V_{\upepsilon}, \mathsf{Z}(\upalpha))
      \ar[rrrr]^{\text{\eqref{eq:gen-integ-1}}}_{\cong}
      & & & & \Hom_{\mathcal{C}}(V, \unitobj).
    }
  \end{equation*}
  We may assume that the right integral $\Lambda^r \in \Hom_{-|A}(\unitobj_{\upepsilon}, \mathsf{Z}(\upalpha))$ corresponds to $\id_{\unitobj}$ via \eqref{eq:gen-integ-1} since $(\unitobj, \Lambda^r) \cong (\unitobj, c \Lambda^r)$ in $\mathcal{I}^r$ for all non-zero $c \in k$. By chasing $\Lambda^r$ in the above diagram, we have $\Lambda^r \circ t_0 = t$. Namely,
  \begin{equation*}
    t_0: (V, t) \to (\unitobj, \Lambda^r)
  \end{equation*}
  is a morphism in $\mathcal{I}^r$. If $t_0': (V, t) \to (\unitobj, \Lambda^r)$ is a morphism in $\mathcal{I}^r$, then
  \begin{equation*}
    \Lambda^r \circ t_0' = t = \Lambda^r \circ t_0.
  \end{equation*}
  Since $\Lambda^r$ is monic, we have $t_0' = t_0$. Thus $(\unitobj, \Lambda^r)$ is a terminal object of $\mathcal{I}^r$.
\end{proof}

The category $\mathcal{I}^{\ell}$ of generalized left integrals of $\mathcal{C}$ is also defined. Let $\Lambda^{\ell}$ be a non-zero left integral of $\mathcal{C}$. In a similar way as above, one can show that $(\unitobj, \Lambda^{\ell})$ is a terminal object of $\mathcal{I}^{\ell}$.

\begin{proof}[Proof of Theorem~\ref{thm:integ-mod-ft}]
  We only prove \eqref{eq:integral-alpha-2}, since the other one, \eqref{eq:integral-alpha-1}, can be proved in a similar way. Since  \eqref{eq:integral-alpha-2} is obviously satisfied if $\Lambda^r = 0$, we may assume that $\Lambda^r$ is a non-zero right integral of $\mathcal{C}$. We then set
  \begin{equation}
    \label{eq:integ-mod-ft-proof-1}
    t = \mathsf{m}_{\unitobj, \upalpha} \circ (\id_A \otimes \Lambda^r).
  \end{equation}
  The morphism $t$ is non-zero, since $\Lambda^r \ne 0$ and
  \begin{equation*}
    t \circ \mathsf{u} = \mathsf{m}_{\unitobj, \upalpha} \circ (u \otimes \id_{\mathsf{Z}(\upalpha)}) \circ \Lambda^r = \id_{\mathsf{Z}(\upalpha)} \circ \Lambda^r = \Lambda^r.
  \end{equation*}
  By the associativity of $\mathsf{m}$, the pair $(\mathsf{A}, t)$ is a generalized right integral. Thus, by Lemma~\ref{lem:gen-integ-terminal}, there exists a unique morphism $\chi: \mathsf{A} \to \unitobj$ in $\mathcal{C}$ such that
  \begin{equation}
    \label{eq:integ-mod-ft-proof-2}
    t = \Lambda_r \circ \chi = \chi \otimes \Lambda^r.
  \end{equation}
  Again by the associativity of $\mathsf{m}$, we have
  \begin{equation*}
    \Lambda^r \circ \chi \circ \mathsf{m}_{\unitobj, \unitobj}
    = t \circ \mathsf{m}_{\unitobj, \unitobj}
    = \Lambda^r \circ (\chi \otimes \chi).
  \end{equation*}
  Since $\Lambda^r$ is monic, we have $\chi \circ \mathsf{m}_{\unitobj, \unitobj} = \chi \otimes \chi$. We also have $\chi \circ \mathsf{u} = \id_{\unitobj}$. Hence the morphism $\chi: A \to \unitobj$ is in fact a morphism of algebras in $\mathcal{C}$.

  By Lemma~\ref{lem:mul-char}, there is an invertible object $\beta \in \mathcal{C}$ such that $\chi = \upchi_{\beta}$. We denote by ${}_{\beta}\unitobj_{\upepsilon}$ the unit object regarded as an $A$-bimodule by the left action given by $\chi = \upchi_{\beta}$ and the right action given by $\upepsilon_{\unitobj}$. Then we have
  \begin{equation}
    \label{eq:integ-mod-ft-proof-3}
    \mathbf{I}' := \Hom_{\mathsf{A}|\mathsf{A}}({}_{\beta}\unitobj_{\upepsilon}, \mathsf{Z}(\upalpha)) \ne 0,
  \end{equation}
  since this space contains $t \ne 0$ (here $\Hom_{\mathsf{A}|\mathsf{A}}$ means the Hom functor of ${}_{\mathsf{A}}\mathcal{C}_{\mathsf{A}}$). Now let $\Phi: \mathcal{C} \boxtimes \mathcal{C} \boxtimes \mathcal{C} \to {}_{\mathsf{A}}\mathcal{C}_{\mathsf{A}}$ be the equivalence given in Lemma~\ref{lem:adj-alg-bimod}. Then
  \begin{equation*}
    \Phi \left(\int_{X \in \mathcal{C}} X \boxtimes \upalpha \boxtimes X^*\right)
    \cong \mathsf{Z}(\upalpha)
    \quad \text{and} \quad
    \Phi(\beta \boxtimes \beta^* \boxtimes \unitobj) \cong {}_{\beta} \unitobj_{\upepsilon}
  \end{equation*}
  as $\mathsf{A}$-bimodules in $\mathcal{C}$; see Lemmas~\ref{lem:adj-alg-bimod-1} and~\ref{lem:adj-alg-bimod-2}. Thus,
  \begin{align*}
    \mathbf{I}'
    & \cong \Hom_{\mathcal{C} \boxtimes \mathcal{C} \boxtimes \mathcal{C}}
      (\beta \boxtimes \beta^* \boxtimes \unitobj,
      \textstyle \int_{X \in \mathcal{C}} X \boxtimes \upalpha \boxtimes X^*) \\
    & \cong \textstyle \int_{X \in \mathcal{C}} 
      \Hom_{\mathcal{C} \boxtimes \mathcal{C} \boxtimes \mathcal{C}}
      (\beta \boxtimes \beta^* \boxtimes \unitobj,
      X \boxtimes \upalpha \boxtimes X^*) \\
    & \cong \textstyle \int_{X \in \mathcal{C}} 
      \Hom_{\mathcal{C}}(\beta, X)
      \otimes_k 
      \Hom_{\mathcal{C}}(\beta^*, \upalpha)
      \otimes_k
      \Hom_{\mathcal{C}}(\unitobj, X^*) \\
    & \cong \Hom_{\mathcal{C}}(\beta^*, \upalpha)
      \otimes_k \left( \textstyle \int_{X \in \mathcal{C}} 
      \Hom_{\mathcal{C}}(\beta, X)
      \otimes_k 
      \Hom_{\mathcal{C}}(\unitobj, X^*) \right).
  \end{align*}
  If $\beta$ is not isomorphic to $\upalpha^*$, then we have $\mathbf{I}' = 0$, which contradicts to~\eqref{eq:integ-mod-ft-proof-3}. Hence $\beta \cong \upalpha^*$. Now equation~\eqref{eq:integral-alpha-2} follows from \eqref{eq:integ-mod-ft-proof-1} and \eqref{eq:integ-mod-ft-proof-2}.
\end{proof}

\subsection{Maschke theorem}

Let $H$ be a finite-dimensional Hopf algebra, and let $\Lambda \in H$ be a non-zero left or right integral of $H$. The Maschke theorem states that $H$ is semisimple (as an algebra over $k$) if and only if $\varepsilon(\Lambda) \ne 0$. This result has been generalized to unimodular finite tensor categories in \cite{2014arXiv1402.3482S}. Using our notion of categorical integrals, we can extend the Maschke theorem to general finite tensor categories:

\begin{theorem}
  \label{thm::Maschke}
  Let $\Lambda$ be a non-zero left or right integral of $\mathcal{C}$. Then
  \begin{equation}
    \label{eq:Maschke}
    \upepsilon_{\upalpha} \circ \Lambda \ne 0
    \iff \text{$\mathcal{C}$ is semisimple}.
  \end{equation}
\end{theorem}
\begin{proof}
  As we have recalled, this theorem has been proved under the assumption that $\mathcal{C}$ is unimodular \cite[Proposition 5.6]{2015arXiv150401178S}. Thus it suffices to consider the case where $\upalpha$ is not isomorphic to $\unitobj$. If this is the case, $\upepsilon_{\upalpha} \circ \Lambda \in \Hom_{\mathcal{C}}(\unitobj, \upalpha)$ must be zero by Schur's lemma. Since a semisimple finite tensor category ($=$ a fusion category) is unimodular \cite{MR2097289}, the finite tensor category $\mathcal{C}$ cannot be semisimple. Thus we have verified~\eqref{eq:Maschke} in the case where $\mathcal{C}$ is not unimodular.
\end{proof}

\subsection{Cosemisimple-unimodularity}

If $H$ is a Hopf algebra, then one can discuss whether $H$ is cosemisimple, {\it i.e.}, semisimple as a coalgebra. It is not clear what a `cosemisimple finite tensor category' means. Nevertheless, using the notion of cointegrals of a finite tensor category, we can define the notion of `cosemisimple unimodular' finite tensor categories as follows:

\begin{definition}
  \label{def:FTC-CU}
  Let $\lambda$ be a non-zero cointegral of $\mathcal{C}$. We say that $\mathcal{C}$ {\em has the property CU} if $\mathcal{C}$ is unimodular ({\it i.e.}, $\upalpha \cong \unitobj$) and the composition
  \begin{equation}
    \label{eq:css-unimo-def}
    \unitobj \xrightarrow{\quad u \quad} \mathsf{Z}(\unitobj)
    \xrightarrow{\quad \cong \quad} \mathsf{Z}(\upalpha)
    \xrightarrow{\quad \lambda \quad} \unitobj
  \end{equation}
  is a non-zero morphism. Here the second arrow in~\eqref{eq:css-unimo-def} is the morphism induced by an isomorphism $\upalpha \cong \unitobj$.
\end{definition}

`CU' is an abbreviation for `cosemisimple unimodular.' The above definition is justified by the following theorem:

\begin{theorem}
  \label{thm:FTC-CU}
  Let $H$ be a finite-dimensional Hopf algebra. Then $H$ is cosemisimple unimodular if and only if $\lmod{H}$ has the property CU.
\end{theorem}
\begin{proof}
  It is enough to consider the case where $H$ is unimodular. Let $\lambda: H \to k$ be a non-zero left cointegral. By the Maschke theorem, $H$ is cosemisimple if and only if $\lambda(1) \ne 0$. By Example~\ref{ex:adj-Hopf-alg} and the discussion in Subsection~\ref{subsec:integ-FTCs-def}, we see that the composition \eqref{eq:css-unimo-def} is the scalar multiplication by $\lambda(1)$. Thus \eqref{eq:css-unimo-def} is a non-zero morphism if and only if $H$ is cosemisimple.
\end{proof}

We say that two finite-dimensional Hopf algebras $H$ and $H'$ are {\em gauge equivalent} if $\lmod{H}$ and $\lmod{H'}$ are equivalent as tensor categories. A quantity or a property of a finite-dimensional Hopf algebra is said to be a {\em gauge invariant} if it is invariant under the gauge equivalence. The above theorem implies:

\begin{corollary}[Aljadeff-Etingof-Gelaki-Nikshych {\cite[Corollary 3.6.]{MR1939116}}]
  The cosemisimple unimodularity is a gauge invariant.
\end{corollary}

\section{Integration functor}
\label{sec:integ-func}

\subsection{Integrals and projective modules}

In this section, we study relations of categorical integrals and morphisms factoring through projective objects. We first recall some related results due to Lorenz \cite{MR1435369,MR1961567}. Let $H$ be a finite-dimensional Hopf algebra over $k$, and let $\Lambda \in H$ be a non-zero left integral. Given a left $H$-module $X$, we define the submodule $I(X)$ of $X$ by
\begin{equation}
  \label{eq:Hopf-integ-functor}
  I(X) = \{ \Lambda x \mid x \in X \}.
\end{equation}
This submodule has interesting properties. For example, there is an isomorphism
\begin{equation*}
  I(W \otimes V^*) \cong \{ f \in \Hom_H(V, W) \mid \text{$f$ factors through a projective $H$-module} \}
\end{equation*}
of vector spaces \cite[Lemma 3.1]{MR1961567}. The dimension of $I(\mathsf{A})$, where $\mathsf{A} \in \lmod{H}$ is the adjoint algebra, is equal to the rank of the Cartan matrix of $H$ reduced modulo the characteristic of $k$ \cite[Theorem 3.4]{MR1435369}.

Our aim is to extend these results to finite tensor categories. In the following, we introduce an analogue of $I(X)$ in terms of categorical left integrals. The submodule $I(X)$ is, by definition, the image of the linear map $X \to X$ given by $x \mapsto \Lambda x$, which is {\em not} $H$-linear in general. From a categorical point of view, this map should be thought of as an $H$-linear map
\begin{equation}
  \label{eq:Hopf-integ-nat-tr}
  X \to X \otimes \upalpha, \quad x \mapsto \Lambda x,
\end{equation}
where $\upalpha$ means the modular function $\upalpha: H \to k$ regarded as a left $H$-module in an obvious way.

\subsection{Integration functor}

We first investigate a categorical analogue of the submodule \eqref{eq:Hopf-integ-functor} and the map \eqref{eq:Hopf-integ-nat-tr} mentioned in the previous subsection. Throughout this section, we fix a finite tensor category $\mathcal{C}$ over $k$ and use the set of notations introduced in Subsection~\ref{subsec:notations}. Let $\upalpha$ be the modular object of $\mathcal{C}$. By Lemmas~\ref{lem:Hom-V-ZW} and~\ref{lem:Hom-V-ZW-br}, there is an isomorphism
\begin{equation*}
  \Hom_{\mathcal{C}}(\unitobj, \mathsf{Z}(\upalpha)) \xrightarrow{\quad \cong \quad} \Nat(\id_{\mathcal{C}}, \id_{\mathcal{C}} \otimes \upalpha)
\end{equation*}
of vector spaces. A categorical analogue of \eqref{eq:Hopf-integ-nat-tr} is defined to be the natural transformation corresponding to a non-zero categorical left integral. Defining $I(X)$ for $X \in \mathcal{C}$ is a bit technical:

\begin{definition}
  \label{def:integ-functor}
  We fix a non-zero categorical left integral $\Lambda \in \Hom_{\mathcal{C}}(\unitobj, Z(\upalpha))$ and define $I(X) \in \mathcal{C}$ for $X \in \mathcal{C}$ to be the image of the morphism
  \begin{equation}
    \label{eq:integ-functor-def}
    X \otimes \upalpha^*
    \xrightarrow{\quad \widetilde{\Lambda}_X \otimes \upalpha^* \quad}
    X \otimes \upalpha \otimes \upalpha^*
    \xrightarrow{\quad \id_X \otimes \coev_{\upalpha}^{-1} \quad}
    X,
  \end{equation}
  where $\widetilde{\Lambda}$ is the natural transformation corresponding to $\Lambda$. Since \eqref{eq:integ-functor-def} is natural in $X$, the assignment $X \mapsto I(X)$ extends to a $k$-linear endofunctor on $\mathcal{C}$ in such a way that the inclusion $I(X) \hookrightarrow X$ is a natural transformation. We call $I: \mathcal{C} \to \mathcal{C}$ the {\em integration functor} of $\mathcal{C}$.
\end{definition}

The object $I(X)$ is defined as a subobject of $X$. Since a non-zero left integral of $\mathcal{C}$ is unique up to scalar multiples, the subobject $I(X) \subset X$ and the functor $I$ do not depend on the choice of the non-zero left integral $\Lambda$.

Till the end of this section, we denote by $\Lambda$ and $\widetilde{\Lambda}$ a fixed non-zero left integral of $\mathcal{C}$ and the corresponding natural transformation $\id_{\mathcal{C}} \to \id_{\mathcal{C}} \otimes \upalpha$, respectively. Since the modular object $\upalpha$ is invertible, we have
\begin{equation}
  \label{eq:integ-IX-zero}
  I(X) = 0 \iff \widetilde{\Lambda}_X = 0
\end{equation}
for $X \in \mathcal{C}$. We also have a canonical isomorphism $I(X) \cong \mathrm{Im}(\widetilde{\Lambda}_X) \otimes \upalpha^*$. Thus there is a functorial epi-mono factorization of $\widetilde{\Lambda}_X$ of the form
\begin{equation}
  \label{eq:integ-epi-mono}
  \widetilde{\Lambda}_X = \left( X
    \twoheadrightarrow I(X) \otimes \upalpha
    \hookrightarrow X \otimes \upalpha
  \right).
\end{equation}

We investigate properties of the integration functor. In the Hopf algebra case, the submodule~\eqref{eq:Hopf-integ-functor} is a direct sum of trivial modules. We say that an object of $\mathcal{C}$ is {\em trivial} if it is the direct sum of finitely many (possibly zero) copies of the unit object. Our first result on the integration functor is:

\begin{lemma}
  \label{lem:integ-functor-triv}
  $I(X)$ is a trivial object for all $X \in \mathcal{C}$.
\end{lemma}
\begin{proof}
  We use the fact that an object $T \in \mathcal{C}$ is trivial if and only if $\uprho^{\ell}_T = \upepsilon_{\unitobj} \otimes \id_T$ \cite[Proposition 5.2]{2015arXiv150401178S}. Let $\xi$ and $\zeta$ be the canonical half-braiding for $\mathsf{Z}(\unitobj)$ and $\mathsf{Z}(\upalpha)$, respectively. By the definition of the tensor product of the Drinfeld center, the half-braiding for $\mathsf{Z}(\unitobj) \otimes \mathsf{Z}(\upalpha)$ is given by
  \begin{equation*}
    \sigma_X = (\xi_{X} \otimes \id_{\mathsf{Z}(\upalpha)}) \circ (\id_{\mathcal{Z}(\unitobj)} \otimes \zeta_X)
  \end{equation*}
  for $X \in \mathcal{C}$. By Lemma \ref{lem:Hom-V-ZW-br}, we have $(\uprho_X^{\ell} \otimes \id_{\upalpha}) \circ (\id_{\mathsf{A}} \otimes \widetilde{\Lambda}_X)$
  \begin{align*}
    & = (((\id_X \otimes \upepsilon_{\unitobj}) \circ \xi_{X}) \otimes \id_{\upalpha})
    \circ (\id_{\mathsf{A}} \otimes ((\id_{X} \otimes \upepsilon_{\upalpha}) \circ \zeta_X \circ (\Lambda \otimes \id_X))) \\
    & = (\id_X \otimes \upepsilon_{\unitobj} \otimes \upepsilon_{\upalpha})
      \circ (\xi_{X} \otimes \id_{\mathsf{Z}(\upalpha)})
      \circ (\id_{\mathcal{Z}(\unitobj)} \otimes \zeta_X)
      \circ (\id_{\mathsf{Z}(\unitobj)} \otimes \Lambda \otimes \id_X) \\
    & = (\id_X \otimes \upepsilon_{\unitobj} \otimes \upepsilon_{\upalpha})
      \circ \sigma_{X} \circ (\id_{\mathsf{Z}(\unitobj)} \otimes \Lambda \otimes \id_X)
  \end{align*}
  for all $X \in \mathcal{C}$. We now note that the equation $\upepsilon_{\unitobj} \otimes \upepsilon_{\upalpha} = \upepsilon_{\upalpha} \circ \mathsf{m}_{\unitobj, \upalpha}$ holds since $\upepsilon$ is monoidal. We also note that $\mathsf{m}_{\unitobj,\upalpha}$ is in fact a morphism in $\mathcal{Z}(\mathcal{C})$. Thus we continue the above computation as follows:
  \begin{align*}
    (\uprho_X^{\ell} \otimes \id_{\upalpha}) \circ (\id_{\mathsf{A}} \otimes \widetilde{\Lambda}_X)
    & = (\id_X \otimes \upepsilon_{\unitobj} \otimes \upepsilon_{\upalpha})
      \circ \sigma_{X} \circ (\id_{\mathsf{Z}(\unitobj)} \otimes \Lambda \otimes \id_X) \\
    & = (\id_X \otimes \upepsilon_{\upalpha})
      \circ (\id_X \otimes \mathsf{m}_{\unitobj, \upalpha})
      \circ \sigma_{X} \circ (\id_{\mathsf{Z}(\unitobj)} \otimes \Lambda \otimes \id_X) \\
    & = (\id_X \otimes \upepsilon_{\upalpha})
      \circ \zeta_X
      \circ (\id_X \otimes \mathsf{m}_{\unitobj, \upalpha})
      \circ (\id_{\mathsf{Z}(\unitobj)} \otimes \Lambda \otimes \id_X).
  \end{align*}
  By the definition of a categorical left integral, we finally obtain
  \begin{equation*}
    (\uprho_X^{\ell} \otimes \id_{\upalpha}) \circ (\id_{\mathsf{A}} \otimes \widetilde{\Lambda}_X)
    = (\id_X \otimes \upepsilon_{\upalpha})
    \circ \zeta_X
    \circ (\upepsilon_{\unitobj}\otimes \Lambda \otimes \id_X)
    = \upepsilon_{\unitobj} \otimes \widetilde{\Lambda}_X.
  \end{equation*}
  Now let $i: I(X) \to X$ be the inclusion morphism. In view of the epi-mono factorization of the form~\eqref{eq:integ-epi-mono}, we have
  \begin{equation*}
    (\uprho_X^{\ell} \otimes \id_{\upalpha}) \circ (\id_{\mathsf{A}} \otimes i \otimes \id_{\upalpha})
    = \upepsilon_{\unitobj} \otimes i \otimes \id_{\upalpha}.
  \end{equation*}
  The left-hand side is equal to $(i \otimes \id_{\upalpha}) \circ (\uprho_{I(X)}^{\ell} \otimes \id_{\upalpha})$ by the naturality of the canonical left action. Since $i$ is monic, and since $\upalpha$ is invertible, we have
  \begin{equation*}
    \uprho_{I(X)}^{\ell} = \varepsilon_{\unitobj} \otimes \id_{I(X)}.
  \end{equation*}
  Hence $I(X)$ is trivial by the fact mentioned at the beginning of the proof.
\end{proof}

Thus $I(X)$ vanishes unless $X$ has a trivial subobject. We are especially interested in the case where $X \in \mathcal{C}$ is injective (or, equivalently, projective). By using the above lemma, we prove:

\begin{lemma}
  \label{lem:integ-inj-full-1}
  {\rm (a)} Let $E \in \mathcal{C}$ be an indecomposable injective object. Then,
  \begin{equation*}
    I(E) \ne 0 \iff \text{$E$ is an injective hull of the unit object $\unitobj$}.
  \end{equation*}
  {\rm (b)} If $i: \unitobj \hookrightarrow E$ is an injective hull of $\unitobj$, then
  \begin{equation*}
    \widetilde{\Lambda}_E = (i \otimes \id_{\upalpha}) \circ p
  \end{equation*}
  for some epimorphism $p: E \twoheadrightarrow \upalpha$ in $\mathcal{C}$.
\end{lemma}
\begin{proof}
  {\rm (a)} Let $V_1, \dotsc, V_n$ be the representatives of isomorphism classes of simple objects of $\mathcal{C}$ with $V_1 = \unitobj$, and let $i: V_s \hookrightarrow E_s$ be an injective hull of $V_s$. Then $\{ E_1, \dotsc E_n \}$ is a complete set of representatives of isomorphism classes of indecomposable injective objects of $\mathcal{C}$. If $I(E_s) = 0$ for all $s$, then, by~\eqref{eq:integ-IX-zero}, we have $\widetilde{\Lambda}_{E_s} = 0$ for all $s$. Since $E = E_1 \oplus \dotsb \oplus E_n$ is an injective cogenerator, we have $\widetilde{\Lambda}_X = 0$ for all objects $X \in \mathcal{C}$. This contradicts to the assumption $\Lambda \ne 0$. Thus $I(E_s) \ne 0$ for some $s \in \{ 1, \dotsc, n \}$. On the other hand, since $V_s$ is a unique simple subobject of $E_s$, we have $I(E_s) = 0$ for all $s = 2, \dotsc, n$. Therefore $I(E_1) \ne 0$.

  {\rm (b)} Let $i: \unitobj \hookrightarrow E$ be an injective hull. Since $I(E) \ne 0$ by (a), and since $\mathrm{Im}(i)$ is the unique non-zero trivial subobject of $E$, we have $I(E) = \mathrm{Im}(i) \cong \unitobj$. Thus we have an epi-mono factorization of the form
  \begin{equation*}
    \widetilde{\Lambda}_E
    = \left( E \xrightarrow{\quad p \quad} \upalpha \xrightarrow{\quad i \otimes \id_{\upalpha} \quad} E \otimes \upalpha \right)
  \end{equation*}
  in view of \eqref{eq:integ-epi-mono}. The proof is done.
\end{proof}

The following lemma is a key observation to give categorical analogues of Lorenz's results mentioned at the beginning of this section.

\begin{lemma}
  \label{lem:Hom-1-IX}
  Let $i: \unitobj \hookrightarrow E$ be an injective hull of $\unitobj$. Then
  \begin{equation}
    \Hom_{\mathcal{C}}(\unitobj, I(X)) = \{ f \circ i \mid f \in \Hom_{\mathcal{C}}(E, X) \}.
  \end{equation}
\end{lemma}
\begin{proof}
  Let $q: E \twoheadrightarrow \upalpha$ be the epimorphism such that $\widetilde{\Lambda}_E = (i \otimes \id_{\upalpha}) \circ q$. Then
  \begin{equation*}
    p := \coev_{\upalpha}^{-1} \circ (q \otimes \upalpha^*): E \otimes \upalpha^* \to \unitobj
  \end{equation*}
  is a projective cover of $\unitobj$. Since $I(X)$ is a trivial object, the map
  \begin{equation}
    \label{lem:Hom-1-IX-pf-1}
    \Hom_{\mathcal{C}}(\unitobj, I(X)) \to \Hom_{\mathcal{C}}(E \otimes \upalpha^*, I(X)),
    \quad x \mapsto x \circ p
  \end{equation}
  is bijective. Let $\widetilde{\Lambda}'_X: X \otimes \upalpha^* \to I(X)$ be the morphism obtained by restricting the codomain of \eqref{eq:integ-functor-def} to its image. Since $E \otimes \upalpha^*$ is projective, the map
  \begin{equation}
    \label{lem:Hom-1-IX-pf-2}
    \Hom_{\mathcal{C}}(E \otimes \upalpha^*, X \otimes \upalpha^*) \to \Hom_{\mathcal{C}}(E \otimes \upalpha^*, I(X)),
    \quad g \mapsto \widetilde{\Lambda}'_X \circ g
  \end{equation}
  induced by~\eqref{eq:integ-functor-def} is surjective. Every element of the set $\Hom_{\mathcal{C}}(E \otimes \upalpha^*, X \otimes \upalpha^*)$ is of the form $f \otimes \upalpha^*$ for some $f: E \to X$. Such a morphism is sent by \eqref{lem:Hom-1-IX-pf-2} to:
  \begin{align*}
    f \otimes \upalpha^*
    & \mapsto (\id_X \otimes \coev_{\upalpha}^{-1}) \circ (\widetilde{\Lambda}_X \otimes \upalpha^*) \circ (f \otimes \upalpha^*) \\
    & = (\id_X \otimes \coev_{\upalpha}^{-1}) \circ (f \otimes \upalpha \otimes \upalpha^*) \circ (\widetilde{\Lambda}_{E} \otimes \upalpha^*)
      \quad \text{(by the naturality of $\widetilde{\Lambda}$)} \\
    & = (\id_X \otimes \coev_{\upalpha}^{-1}) \circ (f \otimes \upalpha \otimes \upalpha^*)
      \circ (i \otimes \upalpha \otimes \upalpha^*)
      \circ (q \otimes \upalpha^*) \\
    & = f \circ i \circ p.
  \end{align*}
  Since the map~\eqref{lem:Hom-1-IX-pf-2} is surjective, we have
  \begin{equation}
    \label{lem:Hom-1-IX-pf-3}
    \Hom_{\mathcal{C}}(E \otimes \upalpha^*, I(X)) = \{ f \circ i \circ p \mid f: E \to X \}.
  \end{equation}
  Now let $f: E \to X$ be a morphism in $\mathcal{C}$. Since $p$ is an epimorphism,
  \begin{equation*}
    \mathrm{Im}(f \circ i) = \mathrm{Im}(f \circ i \circ p) \subset I(X).
  \end{equation*}
  Namely, the morphism $f \circ i$ for $f \in \Hom_{\mathcal{C}}(E, X)$ belongs to the source of the bijection \eqref{lem:Hom-1-IX-pf-1}. Comparing~\eqref{lem:Hom-1-IX-pf-1} with~\eqref{lem:Hom-1-IX-pf-3}, we conclude that the map
  \begin{equation*}
    \Hom_{\mathcal{C}}(E, X) \to \Hom_{\mathcal{C}}(\unitobj, I(X)),
    \quad f \mapsto f \circ i
  \end{equation*}
  is a well-defined surjective map. The proof is done.
\end{proof}

\subsection{Morphisms factoring through projectives}

For two objects $V$ and $W$ of an abelian category $\mathcal{A}$, we define
\begin{equation*}
  \Hom_{\mathcal{A}}^{\mathrm{pr}}(V, W) = \{ f \in \Hom_{\mathcal{A}}(V, W) \mid \text{$f$ factors through a projective object} \}.
\end{equation*}
We give a relation between the set $\Hom_{\mathcal{C}}^{\mathrm{pr}}(V, W)$ for $V, W \in \mathcal{C}$ and the integration functor. As is well-known, there is an isomorphism
\begin{equation}
  \label{eq:duality-iso}
  \Hom_{\mathcal{C}}(\unitobj, W \otimes V^*) \to \Hom_{\mathcal{C}}(V, W),
  \quad t \mapsto (\id_W \otimes \eval_V) \circ (t \otimes \id_V).
\end{equation}
As a categorical analogue of \cite[Lemma 3.1]{MR1961567}, we prove:

\begin{theorem}
  \label{thm:projmorph}
  The isomorphism~\eqref{eq:duality-iso} restricts to the isomorphism
  \begin{equation*}
    \Hom_{\mathcal{C}}(\unitobj, I(W \otimes V^*)) \cong \Hom_{\mathcal{C}}^{\proj}(V, W).
  \end{equation*}
\end{theorem}
\begin{proof}
  Let $t: \unitobj \to W \otimes V^*$ be a morphism in $\mathcal{C}$, and let $f: V \to W$ be a morphism corresponding to $t$. We first suppose that $t$ belongs to $\Hom_{\mathcal{C}}(\unitobj, I(W \otimes V^*))$. Let $i: \unitobj \hookrightarrow E$ be an injective hull of $\unitobj$. Then, as we have seen,
  \begin{equation*}
    \Hom_{\mathcal{C}}(\unitobj, I(W \otimes V^*))
    = \{ g \circ i \mid g \in \Hom_{\mathcal{C}}(E, W \otimes V^*) \}.
  \end{equation*}
  Thus $t = \widetilde{t} \circ i$ for some $\widetilde{t}: E \to W \otimes V^*$. Now $f$ is equal to the composition
  \begin{equation*}
    V \xrightarrow{\quad i \otimes \id_V \quad}
    E \otimes V
    \xrightarrow{\quad \widetilde{t} \otimes \id_V \quad}
    W \otimes V^* \otimes V
    \xrightarrow{\quad \id_W \otimes \eval_V \quad}
    W.
  \end{equation*}
  In particular, $f$ factors through $E \otimes V$, which is projective.

  Conversely, we suppose that $f$ factors through a projective object, that is, there are a projective object $P \in \mathcal{C}$ and morphisms $f': V \to P$ and $f'': P \to W$ in $\mathcal{C}$ such that $f'' \circ f' = f$. We note that $i \otimes \id_V$ is monic. By the injectivity of $P$, there is a morphism $h: E \otimes V \to P$ such that $h \circ (i \otimes \id_V) = f'$. Thus the morphism $t: \unitobj \to W \otimes V^*$ corresponding to $f$ is given by
  \begin{align*}
    t & = (f \otimes \id_{V^*}) \circ \coev_V \\
      & = (f'' \otimes \id_{V^*}) \circ (h \otimes \id_{V^*}) \circ (i \otimes \id_{V} \otimes \id_{V^*}) \circ \coev_V \\
      & = (f'' \otimes \id_{V^*}) \circ (h \otimes \id_{V^*}) \circ (\id_E \otimes \coev_V) \circ i,
  \end{align*}
  and hence $t \in \Hom_{\mathcal{C}}(\unitobj, I(W \otimes V^*))$. The proof is done.
\end{proof}

\subsection{The rank of the Cartan matrix}

\newcommand{\Hig}{\mathrm{Hig}}
\newcommand{\Car}{\mathrm{Car}}
\newcommand{\projcent}{\mathrm{Cent}^{\mathrm{pr}}}

For a finite-dimensional algebra $A$ over $k$, the {\em projective center} of $A$ is defined to be the set of all $A$-bimodules endomorphisms on ${}_A A_A$ factoring through a projective $A$-bimodule. Motivated by this notion, we introduce the following terminology:

\begin{definition}
  The {\em projective center} of a finite abelian category $\mathcal{A}$ is
  \begin{equation*}
    \projcent(\mathcal{A}) := \Hom_{\REX(\mathcal{A})}^{\mathrm{pr}}(\id_{\mathcal{A}}, \id_{\mathcal{A}}),
  \end{equation*}
  where $\REX(\mathcal{A}) := \REX(\mathcal{A}, \mathcal{A})$ is the category of $k$-linear right exact endofunctors on $\mathcal{A}$ (which is a finite abelian category by Lemma~\ref{lem:EW-calc}).
\end{definition}

For the adjoint algebra $\mathsf{A} = \mathsf{Z}(\unitobj)$, there is an isomorphism
\begin{equation}
  \label{eq:adj-alg-end-id}
  \Hom_{\mathcal{C}}(\unitobj, \mathsf{A}) \cong \Nat(\id_{\mathcal{C}}, \id_{\mathcal{C}}) = \Hom_{\REX(\mathcal{C})}(\id_{\mathcal{C}}, \id_{\mathcal{C}})
\end{equation}
by Lemma~\ref{lem:Hom-V-ZW}. We remark that $\projcent(\mathcal{C})$ is a subspace of $\Nat(\id_{\mathcal{C}}, \id_{\mathcal{C}})$.

\begin{theorem}
  \label{thm:integ-higman}
  The isomorphism~\eqref{eq:adj-alg-end-id} restricts to the isomorphism
  \begin{equation*}
    \Hom_{\mathcal{C}}(\unitobj, I(\mathsf{A})) \cong \projcent(\mathcal{C}).
  \end{equation*}
\end{theorem}
\begin{proof}
  We write $\mathcal{E} = \REX(\mathcal{C})$ for simplicity. Let $a: \unitobj \to \mathsf{A}$ be a morphism in $\mathcal{C}$, and let $\widetilde{a} \in \Hom_{\mathcal{E}}(\id_{\mathcal{C}}, \id_{\mathcal{C}})$ be the natural transformation corresponding to $a$. We recall that, by using the left canonical action, $\widetilde{a}$ is given by
  \begin{equation*}
    \widetilde{a}_X = \uprho_X^{\ell} \circ (a \otimes \id_X) \quad (X \in \mathcal{C}).
  \end{equation*}
  We first suppose that $a$ belongs to $\Hom_{\mathcal{C}}(\unitobj, I(\mathsf{A}))$. Let $i: \unitobj \hookrightarrow E$ be an injective hull of $\unitobj$. Then $a = b \circ i$ for some $b: E \to \mathsf{A}$. Thus $\widetilde{a} = a'' \circ a'$ in $\mathcal{E}$, where $a'$ and $a''$ are natural transformations given by
  \begin{equation*}
    a'_X = i \otimes \id_X
    \quad \text{and} \quad
    a''_X = \uprho_X^{\ell} \circ (b \otimes \id_X)
    \quad (X \in \mathcal{C}).
  \end{equation*}
  Namely, the natural transformation $\widetilde{a}$ factors through the functor $E \otimes \id_{\mathcal{C}}$, which is a projective object of $\mathcal{E}$ by Corollary~\ref{cor:Rex-proj}.

  Conversely, we suppose that $\widetilde{a}$ there is a projective object $P \in \mathcal{E}$ and natural transformations $a': \id_{\mathcal{C}} \to P$ and $a'': P \to \id_{\mathcal{C}}$ such that $\widetilde{a} = a'' \circ a'$. Since $\mathcal{E}$ is self-injective by Lemma~\ref{lem:EW-calc}, $P \in \mathcal{E}$ is injective. Now we consider the monomorphism $i \otimes \id: \id_{\mathcal{C}} \to E \otimes \id_{\mathcal{C}}$ in $\mathcal{E}$, where $i: \unitobj \hookrightarrow E$ is the injective hull. By the injectivity of $P$, there is a morphism $\xi: E \otimes \id_{\mathcal{C}} \to P$ in $\mathcal{E}$ such that $a' = \xi \circ (i \otimes \id)$. By Lemma~\ref{lem:Hom-V-ZW}, there is a morphism $b: E \to \mathsf{A}$ such that
  \begin{equation*}
    a''_X \circ \xi_X = (\id_X \otimes \eval_X) \circ (\uppi_{\unitobj; X} \otimes \id_X) \circ (b \otimes \id_X)
  \end{equation*}
  for all $X \in \mathcal{C}$. The right-hand side is equal to $\uprho_X^{\ell}(b \otimes \id_X)$. Thus,
  \begin{equation*}
    \widetilde{a}_X = a''_X \circ a'_X
    = a''_X \circ \xi_X \circ (i \otimes \id_X)
    = \uprho_X^{\ell}(b i \otimes \id_X).
  \end{equation*}
  This means that $a = b i$. Hence $a \in \Hom_{\mathcal{C}}(\unitobj, I(\mathsf{A}))$ by Lemma~\ref{lem:Hom-1-IX}.
\end{proof}

Given a matrix $M$ whose entries are integers, we denote by $\mathrm{rank}_k \, M$ the rank of the matrix $M$ regarded as an matrix over $k$. If $k$ is of characteristic zero, then there is no difference between $\rank_k$ and the usual rank of a matrix. If $k$ is of positive characteristic, then $\mathrm{rank}_k$ may smaller than the usual rank.

Let $V_1, \dotsc, V_n$ be the representatives of isomorphism classes of simple objects of a finite abelian category $\mathcal{A}$, and let $P_s$ be a projective cover of $V_s$. The {\em Cartan matrix}, which we denote by $\Car(\mathcal{A})$, is the $n$-by-$n$ square matrix whose $(i,j)$-entry is the multiplicity of $V_i$ in the composition factors of $P_j$. The following theorem generalizes Lorenz's result \cite[Theorem 3.4]{MR1435369} on the Cartan matrix of the representation category of a finite-dimensional Hopf algebra.

\begin{theorem}
  $\dim_k \Hom_{\mathcal{C}}(\unitobj, I(\mathsf{A})) = \rank_k \Car(\mathcal{C})$.
\end{theorem}
\begin{proof}
  Since $\mathcal{C}$ is Frobenius, there is a Frobenius algebra $F$ over $k$ such that $\mathcal{C}$ is equivalent to $\lmod{F}$ as a $k$-linear category. We choose a $k$-basis $a_1, \dotsc, a_m$ of $F$ and then define $b_1, \dotsc, b_m \in F$ by
  \begin{equation*}
    (a_i, b_j) = \delta_{i,j} \quad (i, j = 1, \dotsc, m),
  \end{equation*}
  where $\delta$ is the Kronecker delta and $(,)$ is the Frobenius form of $F$. The image of the linear map $\tau: F \to F$ given by $\tau(x) = \sum_{i = 1}^m b_i x a_i$ is called the Higman ideal of $F$. By \cite{MR2863463} and \cite{MR2892923}, we have
  \begin{equation*}
    \rank_k \Car(\lmod{F})
    = \dim_k (\text{the Higman ideal of $F$})
    = \dim_k \projcent(\lmod{F}).
  \end{equation*}
  In more detail, the first equality follows from the result on the rank of $\tau$ given in \cite{MR2863463}. The second one follows from the fact that the Higman ideal of a Frobenius algebra coincides with its projective center \cite{MR2892923}. We therefore obtain
  \begin{equation*}
    \rank_k \Car(\mathcal{C}) = \dim_k \projcent(\mathcal{C}).
  \end{equation*}
  The claim of this theorem now follows from Theorem~\ref{thm:integ-higman}.
\end{proof}

\section{Indicators of finite tensor categories}
\label{sec:indicator-ftc}

\subsection{Indicators of Hopf algebras}

Let $H$ be a finite-dimensional Hopf algebra over $k$. For a non-negative integer $m$, the $m$-th Sweedler power map $P^{(m)}: H \to H$ is defined as the $m$-th power of $\id_H \in \End_k(H)$ with respect to the convolution product. Namely, for an element $h \in H$, we have
\begin{equation*}
  P^{(0)}(h) = \varepsilon(h) 1_H
  \quad \text{and} \quad
  P^{(m)}(h) = h_{(1)} h_{(2)} \dotsb h_{(m)}
  \quad (\text{if $m > 0$})
\end{equation*}
with the Sweedler notation. Motivated by the result on the Frobenius-Schur indicators of the regular representation of a semisimple Hopf algebra \cite{MR2213320}, Kashina, Montgomery and Ng defined the $n$-th indicator $\nu_n(H)$ of $H$ by
\begin{equation}
  \label{eq:ind-def-orig}
  \nu_n(H) = \mathrm{Trace}(S \circ P^{(n - 1)}: H \to H)
\end{equation}
for a positive integer $n$ \cite{KMN09}. One of the main results of \cite{KMN09} is the gauge invariance of the indicators: If $H$ and $H'$ are two finite-dimensional Hopf algebras over $k$ such that $\lmod{H}$ and $\lmod{H'}$ are equivalent as tensor categories, then $\nu_n(H) = \nu_n(H')$ for all positive integers $n$.

In view of their result, it is natural to attempt to define the $n$-th indicator of a finite tensor category as a generalization of $\nu_n$. It is difficult to understand the right hand side of \eqref{eq:ind-def-orig} in the categorical context. Fortunately, there is a formula of $\nu_n$ written in terms of integrals of Hopf algebras: If $\lambda: H \to k$ is a non-zero left cointegral and $\Lambda \in H$ is the left integral such that $\lambda(\Lambda) = 1$, then we have
\begin{equation}
  \label{eq:ind-formula-integ}
  \nu_n(H) = \lambda(\Lambda_{(1)} \Lambda_{(2)} \dotsb \Lambda_{(n)})
\end{equation}
for all positive integers $n$ \cite[Corollary 2.6]{KMN09}. In this section, we give a categorical interpretation of the expression $\Lambda_{(1)} \dotsb \Lambda_{(n)}$ in terms of a categorical left integral (Subsection~\ref{subsec:sw-pow-integ}). The result gives us the definition of $\nu_n(\mathcal{C})$ for general finite tensor categories $\mathcal{C}$ (Subsection~\ref{subsec:indicator-FTC}).
 
\subsection{Sweedler powers of integrals}
\label{subsec:sw-pow-integ}

Let $\mathcal{C}$ be a finite tensor category over $k$ with the modular object $\upalpha$, and let $\Lambda$ be a non-zero left integral of $\mathcal{C}$. As we have discussed in Subsection~\ref{subsec:integ-FTCs-def}, the natural transformation $\widetilde{\Lambda}_X: X \to X \otimes \upalpha$ corresponding to $\Lambda$ has an epi-mono factorization of the form
\begin{equation*}
  \widetilde{\Lambda}_X = \left( X \mathop{
      \mathrel{\relbar\joinrel\relbar\joinrel\relbar
        \joinrel\relbar\joinrel\twoheadrightarrow}}^{q_X}
    \, I(X) \otimes \upalpha \,
    \mathop{\mathrel{\lhook\joinrel\xrightarrow{\qquad\qquad}}}^{i_X \otimes \upalpha}
    \, X \otimes \upalpha
  \right),
\end{equation*}
where $I: \mathcal{C} \to \mathcal{C}$ is the integration functor of $\mathcal{C}$ (Definition~\ref{def:integ-functor}). The factorization is functorial, meaning that $q_X$ and $i_X$ are natural in the variable $X$.

We recall from Lemma~\ref{lem:integ-functor-triv} that $I(X)$ is trivial in the sense that it is isomorphic to the direct sum of finitely many $\unitobj$'s. Let $\mathcal{T}$ be the full subcategory of $\mathcal{C}$ consisting of all trivial objects. There is a unique natural isomorphism
\begin{equation*}
  \tau_{T,V}: T \otimes V \to V \otimes T
  \quad (T \in \mathcal{T}, V \in \mathcal{V})
\end{equation*}
such that $\tau_{\unitobj, V} = \id_V$ for all $V \in \mathcal{C}$ ({\it cf}. \cite[Lemma 7.1]{MR2381536}). We now define a categorical analogue of the $n$-th Sweedler power of a left integral:

\begin{definition}
  \label{def:sw-pow-integ}
  \newcommand{\myarrow}[1]{\xrightarrow{\makebox[9em]{$\scriptstyle #1$}}}
  For a positive integer $n$, we define $\Lambda^{[n]} \in \Hom_{\mathcal{C}}(\unitobj, \mathsf{Z}(\upalpha))$ to be the element corresponding to the following natural transformation:
  \begin{align}
    \label{eq:integ-Sw-pow-1}
    \widetilde{\Lambda}_{X}^{[n]} := \Big( X
    & \myarrow{\coev^{(n-1)} \, \otimes \, \id_X}
      \underbrace{Y \otimes \dotsb \otimes Y}_{n-1}
      \, \otimes \, \underbrace{X \otimes \dotsb \otimes X}_n \\
    \label{eq:integ-Sw-pow-2}
    & \myarrow{\id \, \otimes \dotsb \otimes \, \id \, \otimes \, q_{X \otimes \dotsb \otimes X}}
      Y \otimes \dotsb \otimes Y
      \, \otimes \, I(X \otimes \dotsb \otimes X) \otimes \upalpha \\
    \label{eq:integ-Sw-pow-3}
    & \myarrow{\tau \, \otimes \, \id_{\upalpha}}
      I(X \otimes \dotsb \otimes X)
      \, \otimes \, Y \otimes \dotsb \otimes Y \otimes \upalpha \\
    \label{eq:integ-Sw-pow-4}
    & \myarrow{i_{X \otimes \dotsb \otimes X} \, \otimes \, \id \, \otimes \dotsb \otimes \, \id}
      X \otimes \dotsb \otimes X
      \, \otimes \, Y \otimes \dotsb \otimes Y \otimes \upalpha \\
    \label{eq:integ-Sw-pow-5}
    & \myarrow{\id_X \otimes \, \eval^{(n-1)} \otimes \id_{\upalpha}}
      X \otimes \upalpha \, \Big) \qquad (X \in \mathcal{C}).
  \end{align}
  Here, $(Y, \mathsf{e}, \mathsf{d})$ is a right dual object of $X$. The morphisms $\eval^{(m)}$ and $\coev^{(m)}$ are defined inductively by $\eval^{(0)} = \coev^{(0)} = \id_{\unitobj}$,
  \begin{equation*}
    \eval^{(m+1)} = \mathsf{e} \circ (\id_X \otimes \eval^{(m)} \otimes \id_Y),
    \quad
    \coev^{(m+1)} = (\id_Y \otimes \coev^{(m)} \otimes \id_X) \circ \mathsf{d}
  \end{equation*}
  for a non-negative integer $m$.
\end{definition}

It is routine to check that the morphism $\widetilde{\Lambda}^{[n]}_X$  does not depend on the choice of a right dual object of $X$ and is natural in the variable $X$. The morphism $\Lambda^{[n]}$ depend on the choice of $\Lambda$. Since a left integral is unique up to scalar multiples, a different choice of a left integral gives a scalar multiple of $\Lambda^{[n]}$.

The morphism $\widetilde{\Lambda}^{[n]}_X$ can be thought of as an analogue of the $n$-th Frobenius-Schur endomorphism introduced in \cite{MR2381536}. Unlike the case considered in \cite{MR2381536}, we cannot consider the `trace' of $\widetilde{\Lambda}_X$ in our general setting.

\subsection{Indicators of finite tensor categories}
\label{subsec:indicator-FTC}

Let $\mathcal{C}$ be a finite tensor category over $k$. We now define the $n$-th indicator of $\mathcal{C}$ by using Definition~\ref{def:sw-pow-integ}.

\begin{definition}
  \label{def:indicator-FTC}
  Let $\lambda$ be a non-zero cointegral of $\mathcal{C}$, and let $\Lambda$ be the left integral of $\mathcal{C}$ such that $\lambda \circ \Lambda = \id_{\unitobj}$ (which exists by Theorem~\ref{thm:integ-exist}). For a positive integer $n$, we define the $n$-th indicator $\nu_n(\mathcal{C}) \in k$ by
  \begin{equation}
    \nu_n(\mathcal{C}) \, \id_{\unitobj} = \lambda \circ \Lambda^{[n]},
  \end{equation}
  where $\Lambda^{[n]}$ is the morphism defined in the above.
\end{definition}

One can check that $\nu_n(\mathcal{C})$ does not depend on the choice of $\lambda$. Moreover, if $\mathcal{C}$ is the representation category of a Hopf algebra, then $\nu_n(\mathcal{C})$ coincides with the $n$-th indicator of Kashina, Montgomery and Ng:

\begin{theorem}
  \label{thm:indicator-FTC-Hopf}
  If $H$ is a finite-dimensional Hopf algebra, then we have
  \begin{equation*}
    \nu_n(\lmod{H}) = \nu_n(H)
    \quad (n = 1, 2, 3, \dotsc).
  \end{equation*}
\end{theorem}
\begin{proof}
  We use the same notation for $H$ as in Example~\ref{ex:Hopf-integ}. Let $\lambda$ be a non-zero left cointegral of $H$, and let $\Lambda \in H$ be the left integral such that $\lambda(\Lambda) = 1$. As we have seen in Example~\ref{ex:Hopf-integ}, we may regard $\lambda$ and $\Lambda$ as a categorical cointegral and a categorical left integral of $\lmod{H}$, respectively.

  We compute the morphism $\widetilde{\Lambda}^{[n]}_{X}$ for $X \in \lmod{H}$. Let $x_1, \dotsc, x_n$ be a basis of $X$, and let $x^1, \dotsc, x^n$ be the dual basis. With the Einstein notation,
  \newcommand{\myarrow}[1]{\raisebox{-.23ex}{$\vdash$}\hspace{-.5em}
    \xrightarrow{\makebox[4em]{$\scriptstyle\text{#1}$}}}
  \begin{align*}
    x \,\,
    & \myarrow{\eqref{eq:integ-Sw-pow-1}}
      x^{i_1} \otimes \dotsb \otimes x^{i_{n-1}}
      \otimes x_{i_{n-1}} \otimes \dotsb \otimes x_{i_{1}} \otimes x \\
    & \myarrow{\eqref{eq:integ-Sw-pow-2}}
      x^{i_{n-1}} \otimes \dotsb \otimes x^{i_{1}}
      \otimes \Lambda_{(1)} x_{i_{1}} \otimes \dotsb \otimes \Lambda_{(n-1)} x_{i_{n-1}} \otimes \Lambda_{(n)} x \\
    & \myarrow{\eqref{eq:integ-Sw-pow-3},\eqref{eq:integ-Sw-pow-4}}
      \Lambda_{(1)} x_{i_{1}} \otimes \dotsb \otimes \Lambda_{(n-1)} x_{i_{n-1}} \otimes \Lambda_{(n)} x
      \otimes x^{i_{n-1}} \otimes \dotsb \otimes x^{i_{1}} \\
    & \myarrow{\eqref{eq:integ-Sw-pow-5}}
      \Lambda_{(1)} x_{i_{1}}
      \cdot x^{i_1}(\Lambda_{(2)} x_{i_2})
      \dotsb x^{i_{n-2}}(\Lambda_{(n-1)} x_{i_{n-1}})
      \cdot x^{i_{n-1}}(\Lambda_{(n)} x)
  \end{align*}
  for $x \in X$. Using the identity $x^i(x) \cdot x_i = x$, we get
  \begin{equation*}
    \widetilde{\Lambda}_X^{[n]}(x) = \Lambda_{(1)} \dotsb \Lambda_{(n)} x
  \end{equation*}
  for $x \in X$. Thus $\Lambda^{[n]} \in \Hom_H(k, \mathsf{Z}(\upalpha))$ is the element $\Lambda_{(1)} \dotsb \Lambda_{(n)} \in H$ if we view a linear map $k \to \mathsf{Z}(\alpha)$ as an element of $H$. Thus, by~\eqref{eq:ind-formula-integ}, we have
  \begin{equation*}
    \nu_n(\lmod{H}) = \lambda(\Lambda^{[n]}) = \nu_n(H). \qedhere
  \end{equation*}
\end{proof}

It is routine to check $\nu_n(\mathcal{C}) = \nu_n(\mathcal{D})$ for all positive integers $n$ when $\mathcal{C}$ and $\mathcal{D}$ are equivalent finite tensor categories. Thus we obtain:

\begin{corollary}[Kashina-Montgomery-Ng {\cite{KMN09}}]
  For all positive integer $n$, the $n$-th indicator a finite-dimensional Hopf algebra is a gauge invariant.
\end{corollary}

\bibliographystyle{alpha}
\def\cprime{$'$}

\end{document}